\documentclass[11pt]{article}
\input{amssym}
\input{epsf}
\newtheorem{prethm}{{\bf Theorem}}

\newenvironment{theorem}{\begin{prethm}{\hspace{-0.5
em}{\bf.}}}{\end{prethm}}
\newtheorem{preobs}{{\bf Observation}}

\newenvironment{obs}{\begin{preobs}{\hspace{-0.5
em}{\bf.}}}{\end{preobs}}
\newtheorem{preex}{{\bf Example}}

\newtheorem{prelem}{{\bf Theorem}}

\newenvironment{lem}{\begin{prelem}{\hspace{-0.5
em}{\bf.}}}{\end{prelem}}
\newtheorem{prepro}{{\bf Proposition}}

\newenvironment{pro}{\begin{prepro}{\hspace{-0.5
em}{\bf.}}}{\end{prepro}}
\newtheorem{preprop}{{\bf Proposition}}

\newtheorem{preques}{{\bf Question}}

\newtheorem{precor}{{\bf Corollary}}

\newenvironment{cor}{\begin{precor}{\hspace{-0.5
em}{\bf.}}}{\end{precor}}
\newtheorem{prefact}{{\bf Fact}}

\newenvironment{fact}{\begin{prefact}{\hspace{-0.5
em}{\bf.}}}{\end{prefact}}
\newtheorem{prelemma}{{\bf Lemma}}

\newenvironment{lemma}{\begin{prelemma}{\hspace{-0.5
em}{\bf.}}}{\end{prelemma}}
\newtheorem{presol}{{\bf Solution}}

\newtheorem{preproof}{{\bf Proof.}}

\newenvironment{proof}[1]{\begin{preproof}{\rm
               #1}\hfill{$\rule{2mm}{2mm}$}}{\end{preproof}}
\begin{document}
\title{
{\Large{\bf Characterization of $\bf n$-Vertex Graphs with Metric
Dimension $\bf n-3$}}}
%

{\small
\author{
{\sc Mohsen Jannesari} and {\sc Behnaz Omoomi  }\\
[1mm]
{\small \it  Department of Mathematical Sciences}\\
{\small \it  Isfahan University of Technology} \\
{\small \it 84156-83111, \ Isfahan, Iran}}




 \maketitle \baselineskip15truept

\begin{abstract}
For an ordered set $W=\{w_1,w_2,\ldots,w_k\}$ of vertices and a
vertex $v$ in a connected graph $G$, the ordered $k$-vector
$r(v|W):=(d(v,w_1),d(v,w_2),\ldots,d(v,w_k))$ is  called  the
(metric) representation of $v$ with respect to $W$, where $d(x,y)$
is the distance between the vertices $x$ and $y$. The set $W$ is
called  a resolving set for $G$ if distinct vertices of $G$ have
distinct representations with respect to $W$. The minimum
cardinality of a resolving set for $G$ is its metric dimension. In
this paper, we characterize all graphs of order $n$ with metric
dimension  $n-3$.
\end{abstract}


\section{Introduction}

Throughout this paper $G=(V,E)$ is a finite, simple, and connected
graph of order $n(G)$. The distance between two vertices $u$ and
$v$, denoted by $d_G(u,v)$, is the length of a shortest path
between $u$ and $v$ in $G$. We write it simply $d(u,v)$ when no
confusion can arise. Also, the diameter of $G$,
$max_{\{u,v\}\subseteq V(G)}d(u,v)$, is denoted by $diam(G)$. The
vertices of a connected graph can be represented by different
ways,  for example, the vectors which theirs components are the
distances between the vertex and the vertices in a given subset of
vertices.

For an ordered set $W=\{w_1,w_2,\ldots,w_k\}\subseteq V(G)$ and a
vertex $v$ of $G$, the  $k$-vector
$$r(v|W):=(d(v,w_1),d(v,w_2),\ldots,d(v,w_k))$$
is called  the  (metric) {\it representation}  of $v$ with respect
to $W$. The set $W$ is called a {\it resolving set} (locating set)
for $G$ if distinct vertices have different representations. In
this case we say the set $W$ {\it resolves} $G$. Elements in a
resolving set are called {\it landmarks}. A resolving set $W$ for
$G$ with minimum cardinality is  called a {\it basis} of $G$, and
its cardinality is the {\it metric dimension} of $G$, denoted by
$\beta(G)$.

The concept of (metric) representation is introduced by
Slater~\cite{Slater1975} (see~\cite{Harary}). For more results in
this concept see~\cite{cameron,trees,cartesian
product,bounds,sur1,Discrepancies}. \\
 It is obvious that for every graph $G$ of order $n$, $1\leq \beta(G)\leq
 n-1$. Yushmanov~\cite{Yushmanov} improved this bound to $\beta(G)\leq
 n-diam(G)$.
  Khuller et al.~\cite{landmarks} and Chartrand et al.~\cite{Ollerman}
  independently proved that $\beta(G)=1$ if and only
if $G$ is a path. Also, all graphs with metric dimension two are
characterized   by Sudhakara and Hemanth Kumar~\cite{dim=2}.
 Chartrand et al.~\cite{Ollerman} proved that the
only graph of order $n$, $n\geq 2$,  with metric dimension $n-1$
is the complete graph $K_n$. They also provided a
 characterization of  graphs of order $n$ and
metric dimension $n-2$.   In~\cite{idea} the problem of
characterization of all graphs of order $n$ with metric dimension
$n-3$ is proposed.  In this paper, we answer to this question and
provide a characterization of graphs with metric dimension $n-3$.
First in next section, we present some definitions and known
results which are necessary to prove our main theorem.

\section{Preliminaries}
In this section, we present some definitions and known results
which are necessary  to prove our main theorem. The notations
$u\sim v$ and $u\nsim v$ denote the adjacency and none-adjacency
between vertices $u$ and $v$, respectively. An edge with end
vertices $u$ and $v$ is denoted by $uv$. A path of order $n$,
$P_n$, and a cycle of order $n$, $C_n$, are denoted by
$(v_1,v_2,\ldots, v_n)$ and $(v_1,v_2,\ldots,v_n,v_1)$,
respectively.
\par We say an ordered set
$W$ {\it resolves} a set $T$ of vertices in $G$, if the
representations of vertices in $T$ are distinct with respect to
$W$. When $W=\{x\}$, we say that the vertex $x$ resolves $T$. To
see that whether a given set $W$ is a resolving set for $G$, it is
sufficient to look at the representations of vertices in
$V(G)\backslash W$, because
$w\in W$ is the only vertex of $G$ for which $d(w,w)=0$.\\

In~\cite{Ollerman} all graphs of order $n$ with metric dimension
 $n-2$ are characterized as follows.
\begin{lem} {\rm\cite{Ollerman}} \label{n-2}
Let $G$ be a  graph of order $n\geq 4$. Then $\beta(G)=n-2$ if and
only if $G=K_{s,t}~(s,t\geq 1), G=K_s\vee\overline K_t~(s\geq 1,
t\geq 2)$, or $G=K_s\vee (K_t\cup K_1)~ (s,t\geq 1)$.
\end{lem}
For each vertex $v\in V(G)$, let $\Gamma_i(v):=\{u\in
V(G)~|~d(u,v)=i\}$. Two distinct vertices $u,v$ are {\it twins} if
$\Gamma_1(v)\backslash\{u\}=\Gamma_1(u)\backslash\{v\}$. It is
called that $u\equiv v$ if and only if $u=v$ or $u,v$ are twins.
In~\cite{extermal}, it is proved that the relation $\equiv$ is an
equivalent relation. The equivalence class of the vertex $v$ is
denoted by $v^*$. The {\it twin graph} of $G$, denoted by $G^*$,
is the graph with vertex set $V(G^*):=\{v^*~|~v\in V(G)\}$, where
$u^*v^*\in E(G^*)$ if and only if $uv\in E(G)$. It is easy to
prove that $u,v$ are adjacent in $G$ if and only if all vertices
of $u^*$ are adjacent to all vertices of $v^*$, hence the
definition of $G^*$ is well defined. For each subset $S\subseteq
V(G)$, let $S^*$ denote the set $\{v^*\in V(G^*)~|~v^*\subseteq
S\}$. Also, by $\Gamma_i^*(v^*)$, we mean the set $\{u^*\in
V(G^*)~|~d_{G^*}(u^*,v^*)=i\}$. For each $v\in v^*$, it is
immediate that $\Gamma^*_i(v^*)=\Gamma^*_i(v)$ if $i\neq 0$.
Furthermore, we define \begin{description} \item$R_1(v):=\{x\in
\Gamma_1(v)~|~\exists y\in \Gamma_2(v):x\sim y\}$, \item
$R_2(v):=\Gamma_1(v)\backslash R_1(v)$,\item $R^*_1(v^*):=\{x^*\in
\Gamma^*_1(v^*)~|~\exists y^*\in \Gamma^*_2(v^*):x^*\sim
y^*\}$,\item $R^*_2(v^*):=\Gamma^*_1(v^*)\backslash R^*_1(v^*)$.
\end{description}

\par We say a set $S$ of vertices is  {\it homogeneous}
if the induced subgraph  by $S$ in  $G$, $G[S]$, is a complete or
an empty subgraph of $G$. In this terminology, it is proved
in~\cite{extermal} that each vertex of $G^*$ is a homogeneous
subset of $V(G)$.
\begin{obs} {\rm\cite{extermal}}\label{twins}
If the vertices $u,v$ are twins in a graph $G$ and $S$ resolves
$G$, then $u$ or $v$ is in $S$. Moreover, if $u\in S$ and $v\notin
S$, Then $(S\setminus\{u\})\cup \{v\}$ also resolves $G$.
\end{obs}
\begin{pro} {\rm\cite{extermal}}\label{diam G^*}
Let $G\neq K_1$ be a  graph. Then $diam(G^*)\leq diam(G).$
Moreover, if $u,v\in V(G)$ are not twin vertices of $G$, then
$d_{G^*}(u^*,v^*)=d_G(u,v)$.
\end{pro}

As in ~\cite{extermal}, we say that $v^*\in V(G^*)$ is of type:
\begin{description}
\item $\bullet$ (1) if $|v^*|=1$, \item$\bullet$ (K) if
$G[v^*]\cong K_r$ and $r\geq 2$, \item $\bullet$ (N) if
$G[v^*]\cong \overline{K_r}$ and $r\geq 2$.
\end{description}
A vertex of $G^*$ is of type (1K) if it is of type (1) or (K). A
vertex is of type (1N) if it is of type (1) or (N). A vertex is of
type (KN) if it is of type (K) or (N).  We denote by $\alpha(G^*)$
the number of vertices of $G^*$ of type (K) or (N). It is obvious
that $G$ is uniquely determined by $G^*$, and the type and
cardinality of each vertex of $G^*$.

\par Hernando et al.~\cite{extermal} characterized all
graphs of order $n$, diameter $d$ and metric dimension $n-d$ by
the following theorem.
\begin{lem} {\rm\cite{extermal}}\label{n-d}
Let $G$ be a  graph of order $n$ and diameter $d\geq 3$. Let $G^*$
be the twin graph of $G$. Then $\beta(G)=n-d$ if and only if $G^*$
is one of the following graphs:
\begin{description}
\item 1. $G^*\cong P_{d+1}$ and one of the following cases holds\\
(a) $\alpha(G^*)\leq 1$;\\ (b) $\alpha(G^*)=2$, the two vertices
of $G^*$ not of type (1) are adjacent, and if one is a leaf of
type (K), then the other is also of type (K);\\
(c) $\alpha(G^*)=2$, the two vertices of $G^*$ not of type (1) are
at distance $2$ and both are of type (N); or\\
(d) $\alpha(G^*)=3$ and there is a vertex of type (KN) adjacent to
two vertices of type (N). \item 2. $G^*\cong P_{d+1,k}$ {\rm(}the
path $(u^*_0,u^*_1,\ldots,u^*_d\rm)$ with one extra vertex
adjacent to $u^*_{k-1}${\rm)} for some integer $k\in [3,d-1]$, the
degree-$3$ vertex $u^*_{k-1}$ of $G^*$ is of any type, each
neighbor of $u^*_{k-1}$ is of type (1N), and every other vertex is
of type (1). \item 3. $G^*\cong P^\prime_{d+1,k}$ {\rm(}the path
$(u^*_0,u^*_1,\ldots,u^*_d)$ with one extra vertex adjacent to
$u^*_{k-1}$ and $u^*_k${\rm)} for some integer $k\in [2,d-1]$, the
three vertices in the cycle are of type (1K), and every other
vertex is of type (1).
\end{description}
\end{lem}
To prove our main theorem, we need the following propositions.
\begin{pro}\label{sugraph}
Let  $H$ be a subgraph of a graph $G$. If $\beta(H)=n(H)-t$ and
$d_H(u,v)=d_G(u,v)$ for all pairs of vertices in $H$, then
$\beta(G)\leq n(G)-t$.
\end{pro}
\begin{proof} {Let $W$ be a basis of $H$. Since $d_H(u,v)=d_G(u,v)$ for
all pairs of vertices in $H$, $W$ as a subset of $V(G)$ resolves
all vertices of $H$. Hence, $T=W\cup (V(G)\backslash V(H))$ is a
resolving set for $G$ where, $|T|=n(G)-t$. Therefore, the metric
dimension of $G$ is at most $n(G)-t$. }\end{proof}
\begin{cor}\label{sugraph d=2}
If $H$ is an induced subgraph of a graph $G$, where $diam(H)=2$,
and $\beta(H)=n(H)-t$ for some positive integer $t$, then
$\beta(G)\leq n(G)-t$.
\end{cor}
\begin{cor}\label{sugraphH<K<G}
Let $H$ be an induced subgraph of a graph $G$, and $G$ be an
induced subgraph of a graph $R$, where $diam(H)=diam(G)=2$,
$\beta(R)=n(R)-t$, and $\beta(H)=n(H)-t$ for some positive integer
$t$. Then $\beta(G)=n(G)-t$.
\end{cor}
\begin{proof}{
Let $\beta(G)=n(G)-s$, for some positive integer $s$. By
Corollary~\ref{sugraph d=2}, we have $n(G)-s=\beta(G)\leq n(G)-t$
and $n(R)-t=\beta(R)\leq n(R)-s$. Therefore,  $s=t$. }\end{proof}
\begin{pro}\label{dim G^*<n-t}
Let $G$ be  a  graph  and $G^*$ be the twin graph of $G$. If
$\beta(G^*)=n(G^*)-t$ for some positive integer $t$, then
$\beta(G)\leq n(G)-t$.
\end{pro}
\begin{proof}{ Let $S^*$ be a basis of $G^*$ and $T^*=V(G^*)\backslash S^*$.
We choose a vertex $v$ from each $v^*\in T^*$ and let
$T=\{v~|~v^*\in T^*\}$. Since $S^*$ is a basis of $G^*$, for each
pair of vertices $u^*,v^*\in T^*$, there exist a vertex $x^*\in
S^*$ such that $d_{G^*}(x^*,u^*)\neq d_{G^*}(x^*,v^*)$. Note that
neither $u$ nor $v$ is twin with $x$, for each $x\in x^*$.
Therefore, by Proposition~\ref{diam G^*}, we have
$d_{G^*}(x^*,u^*)=d_G(x,u)$ and $d_{G^*}(x^*,v^*)=d_G(x,v)$.
Hence, $d_G(x,u)\neq d_G(x,v)$, which implies that the set
$S=\bigcup_{v^*\in S^*}v^*$ resolves $T$. Therefore,
$V(G)\setminus T$ is a resolving set for $G$ of cardinality
$n(G)-t$, thus $\beta(G)\leq n(G)-t$. }\end{proof}
\begin{pro}\label{B(G)>n-n*}
Let $G$ be  a  graph  and $G^*$ be the twin graph of $G$. Then
$\beta(G)\geq n(G)-n(G^*)$.
\end{pro}
\begin{proof}{
If $S$ resolves $G$, then Observation~\ref{twins} shows that $S$
contains at least $|v^*|-1$ vertices from each $v^*\in V(G^*)$.
Hence, $\beta(G)\geq n(G)-n(G^*)$. }\end{proof}
\section{\label{sec:sts}Main Results}
Let $G$ be a connected graph of order $n$ and metric dimension
$n-3$. Since $n-3=\beta(G)\leq n-diam(G)$, it follows that
$diam(G)\leq 3$. If $diam(G)=1$, then $G\cong K_n$, contrary to
$\beta(G)=n-3$. If $diam(G)=3$, then in Theorem~\ref{n-d}, let
$d=3$, which obtains a characterization of graphs $G$ with
$\beta(G)=n-3$, where $diam(G)=3$, (Note that in this case the
interval $[3,2]$ is empty, hence, Case 2 dose not occur).
Therefore, it is enough to consider the case $diam(G)=2$. The
following theorem is our main result, which is a characterization
of all graphs with metric dimension $n-3$ and diameter $2$.
\begin{theorem} \label{d=2}
Let $G$ be a  graph of order $n$ and diameter $2$ and $G^*$ be the
twin graph of $G$. Then $\beta(G)=n-3$ if and only if $G^*$
satisfies in one of the following structures:
\begin{description}
\item $G_1$: $G^*\cong K_3$ and has at most one vertex of type
(1K);
\item $G_2$: $G^*\cong P_{3}$ and one of the following cases holds:\\
(a) The degree-$2$ vertex is of type (N) and one of the leaves is of type (K)
and the other is of any type;\\
(b) One of the leaves is of type (K), the other is of type (KN)
and the degree-$2$ vertex is of any type; \item $G_3$: $G^*$ is
the paw {\rm(}a triangle with a pendant edge{\rm)}, and the
degree-$3$ vertex is of any type, one of the degree-$2$ vertices
is of type (N), the other is of type (1K), the leaf is of type
(1N). Moreover, a degree-$2$ vertex of type (K) yields the leaf
and the degree-$3$ vertex are not of type (N); \item $G_4$:
$G^*\cong C_5$, and each vertex is of type (1); \item $G_5$: $G^*$
is $C_5$ with a chord, and the adjacent degree-$2$ vertices are of
type (1) and the other vertices are of type (1K); \item $G_6$:
$G^*$ is a $C_5$ with two adjacent chords. The degree-$4$ vertex
is of any type, the others are of type (1K). Furthermore, two
non-adjacent vertices are not of type (K), and two adjacent
vertices are not of different types (K) and (N); \item $G_7$:
$G^*$ is a kite with a pendant edge adjacent to a degree-$3$
vertex, and the leaf is of type (1), the degree-$4$ and degree-$3$
vertices are of type (1K), one of the degree-$2$ vertices is of
type (K) and the other is of type (1). \item $G_8$: $G^*$ is the
kite, and one of the degree-$2$ vertices is of type (K) the other
is of type (1), one of the degree-$3$ vertices is of type (N), and
the other is of type (1K);\item $G_9$: $G^*\cong C_4$, and two
adjacent vertices are of type (K) and others are of type (1);\item
$G_{10}$: $G^*\cong C_4\vee K_1$, and two degree-$3$ adjacent
vertices are of type (K), degree-$4$ vertex is of type (1K), and
others are of type (1).
\end{description}
In Figure~\ref{G^*and dim=n-3} the scheme of the above $10$
structures are shown.
\end{theorem}
\begin{figure}[ht]
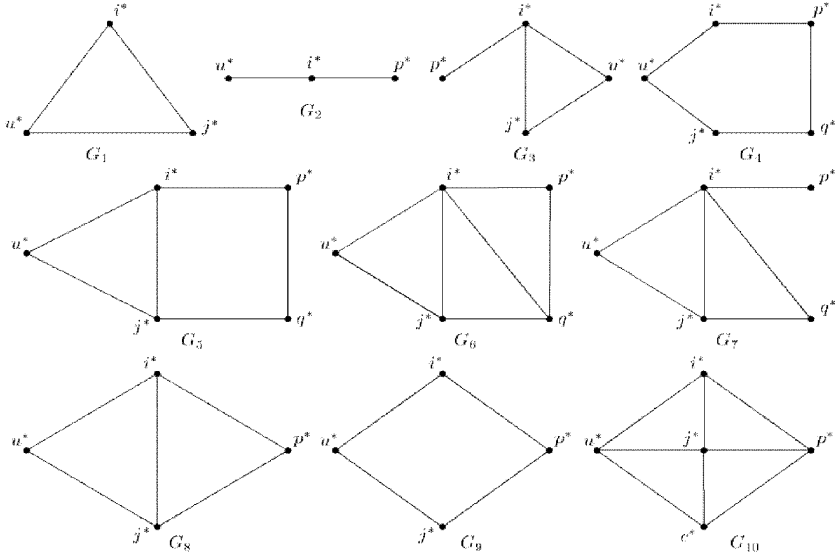
\hspace{1.5cm}
\vspace*{5cm}\special{em:graph fig1.bmp} \vspace*{1.5cm}
\vspace{1.3cm}
\caption{\label{G^*and
dim=n-3}{\scriptsize The twin graphs of graphs with diameter $2$
and metric dimension $n-3$.} }
\end{figure}
\subsection{\label{necessity} Proof of Necessity} Throughout  this
section, $G$ is a  graph of order $n$, diameter $2$, metric
dimension $n-3$, and $G^*$ is the twin graph of $G$. Note that,
Proposition~\ref{diam G^*} implies that $diam(G^*)\leq 2$. Through
a sequence of lemmas and propositions, we show that $G^*$ has one
of the structures $G_1$ to $G_{10}$.
\begin{pro}\label{G^*=K_n pro}
If $diam(G^*)=1$, then $G^*$ satisfies in structure $G_1$.
\end{pro}
\begin{proof}{
Let $u^*,v^*$ be two vertices of $G^*$ of type (1K). Since $G^*$
is a complete graph, every pair of vertices $u\in u^*$ and $v\in
v^*$ of $G$ are twins, thus $u^*=v^*$. Hence, $G^*$ has at most
one vertex of type (1K). If vertices $v^*_1,v^*_2,v^*_3$, and
$v^*_4$ of $G^*$ (except possibly $v^*_1$) are of type (N), then
we choose an arbitrary vertex $v_i\in v^*_i$, for each $i$, $1\leq
i\leq 4$, and $u_i\in v^*_i\backslash\{v_i\}$, for each $i$,
$2\leq i\leq 4$. Let $T=\{u_2,u_3,u_4\}$. It follows
that$$r(v_1|T)=(1,1,1),\quad r(v_2|T)=(2,1,1),\quad
r(v_3|T)=(1,2,1),\quad r(v_3|T)=(1,1,2).$$ Hence, the set
$V(G)\backslash\{v_1,v_2,v_3,v_4\}$ is a resolving set for $G$ and
$\beta(G)\leq n-4$. This contradicts our assumption,
$\beta(G)=n-3$. Therefore, $G^*$ has at most three vertices.
Assume $G^*\cong K_t$ for some  integer $t\in [1,3]$. If $t=1$,
then $G\cong K_n$ or $G \cong\overline K_n$. If $t=2$, then
$G\cong K_{r,s}$ or $G\cong K_r \vee \overline K_s$. Since
$\beta(G)=n-3$, the above cases are impossible. Consequently
$G^*\cong K_3$, which is the desired conclusion. }\end{proof} The
remainder of this section will be devoted to the case
$diam(G^*)=2$. It is clear that in this case, there exists a
vertex $v\in V(G)$ such that $\Gamma^*_2(v^*)\neq \emptyset$ and
$\Gamma_2(v)\neq\emptyset$.
\begin{lemma}\label{G^*=K_3}
If \- $\Gamma_2^*(v^*)\neq\emptyset$ and $v\in v^*$, then
$\bigcup_{u^*\in \Gamma_i^*(v^*)}u^*  \subseteq \Gamma_i(v)$,
where $i\in\{1,2\}$, $\bigcup_{u^*\in R_2^*(v)}u^*
\linebreak\subseteq R_2(v)$, and $\bigcup_{u^*\in
R_1^*(v)}u^*=R_1(v)$. Moreover, if $v^*$ is of type (1K), then
$\bigcup_{u^*\in \Gamma_2^*(v^*)}u^*=\Gamma_2(v)$,
$R_1^*(v^*)=R_1^*(v)$, and $R_2^*(v^*)=R_2^*(v)$.
\end{lemma}
\begin{proof}{  It is clear that a vertex in
$\Gamma_1(v)$ is not twin with any vertex of $\Gamma_2(v)$.
Therefore, all twins of vertices in $\Gamma_1(v)$ are in the set
$\Gamma_1(v)\cup \{v\}$ and all twins of vertices in $\Gamma_2(v)$
are in $\Gamma_2(v)\cup \{v\}$. This gives $\Gamma_i(v)\backslash
v^*=\bigcup_{u^*\in \Gamma_i^*(v^*)}u^*$, for $i\in\{1,2\}$.
Hence, $\bigcup_{u^*\in \Gamma_i^*(v^*)}u^*\subseteq \Gamma_i(v)$,
when $i\in\{1,2\}$. Note that all twins of vertices in $R_1(v)$
are in $R_1(v)$, because each member of $R_1(v)$ is adjacent to
$v$ and has at least one neighbor in $\Gamma_2(v)$ while the
vertices in $R_2(v)\cup\{v\}$ have not such neighbors. Thus
$\bigcup_{u^*\in R_1^*(v)}u^*=R_1(v)$. By the same reason, all
twins of vertices in $R_2(v)$ are in $R_2(v)\cup\{v\}$.
Consequently, $\bigcup_{u^*\in R_2^*(v)}u^*=R_2(v)\backslash
v^*\subseteq R_2(v)$.
\par Now let $v^*$ is of type (1K). Therefore, $v$ can only be
twin with the vertices of $R_2(v)$. Hence, $\bigcup_{u^*\in
\Gamma_2^*(v^*)}u^*=\Gamma_2(v)$. It is clear that
$R^*_1(v^*)\subseteq R^*_1(v)$. If there exist a vertex $u^*\in
R^*_1(v)\backslash R^*_1(v^*)$, then $u^*\in R^*_2(v^*)$, because
$\Gamma^*_1(v^*)=R^*_1(v^*)\cup R^*_2(v^*)$. Since $v^*\subseteq
R_2(v)\cup \{v\}$, all neighbors of each vertex $u\in u^*$ are in
the set $\Gamma_1(v)\cup\{v\}$, this contradicts the fact that
$u\in R_1(v)$. Therefore, $R_1^*(v^*)=R_1^*(v)$ and consequently,
$R_2^*(v^*)=R_2^*(v)$.
  }\end{proof}
\begin{lemma}\label{one homogeneous}
For each $v$, where $\Gamma_2(v)\neq\emptyset$, at least one of
the sets $\Gamma_1(v)$ and $\Gamma_2(v)$ is homogeneous.
\end{lemma}
\begin{proof}{
Let $\Gamma_2(v)\neq\emptyset$. On the contrary, suppose that both
$\Gamma_1(v)$ and $\Gamma_2(v)$ are not homogeneous. Therefore,
there exist vertices $v_1, v_2$, and $v_3$ in $\Gamma_1(v)$, and
vertices $u_1, u_2$, and $u_3$ in $\Gamma_2(v)$ such that $v_1\sim
v_2, v_2\nsim  v_3$ and $u_1\sim u_2$, $u_2\nsim  u_3$. If
$W^\prime=\{v,v_2,u_2\}$, then the representations of
$v_1,v_3,u_1$ and $u_3$ with respect to $W^\prime$ are as follows
$$r(v_1|W^\prime)=(1,1,*),\quad r(v_3|W^\prime)=(1,2,*),\quad
r(u_1|W^\prime)=(2,*,1),\quad r(u_3|W^\prime)=(2,*,2),$$where $*$
is $1$ or $2$. These four representations are distinct, hence,
$V(G)\backslash\{v_1,v_3,u_1,u_3\}$ is a resolving set for $G$.
Thus $\beta(G)\leq n-4$, which is a contradiction. Therefore, at
least one of the sets $\Gamma_1(v)$ or $\Gamma_2(v)$ is
homogeneous. }\end{proof} By Lemma~\ref{one homogeneous}, to
complete the proof of necessity, we need to consider the following
two  cases.
\vspace{4mm} \\{\bf Case 1.} There exist a vertex $v\in V(G)$ such
that $\Gamma^*_2(v^*)\neq\emptyset$ and $\Gamma_1(v)$ is
homogeneous. \vspace{4mm}\\ By the assumption of Case 1, the
following results are obtained.
\begin{fact}\label{R_1^*<2}
$|R^*_1(v^*)|\leq 2$.
\end{fact}
\begin{proof}{ Since every vertex of $R^*_1(v^*)$ has a neighbor in
$\Gamma^*_2(v^*)$ and $\Gamma_1(v)$ is homogeneous, for distinct
vertices $x^*,y^*\in R^*_1(v^*)$, the sets $\Gamma^*_1(x^*)\bigcap
\Gamma^*_2(v^*)$ and $\Gamma^*_1(y^*)\bigcap \Gamma^*_2(v^*)$ are
distinct nonempty sets. Therefore, the set $\Gamma^*_2(v^*)$
resolves the vertices of $R^*_1(v^*)$ in $G^*$. Moreover, since
every vertex of $R^*_1(v^*)$ has a neighbor in $\Gamma^*_2(v^*)$
and $v^*$ has not such a neighbor, if we compute the
representations of the vertices in $R^*_1(v^*)\cup\{v^*\}$ with
respect to $\Gamma^*_2(v^*)$, then the representation of each
vertex in $R^*_1(v^*)$ has a coordinate 1 while all coordinates of
$r(v^*|\Gamma^*_2(v^*))$ are 2. Therefore, $\Gamma^*_2(v^*)$
resolves the set $R^*_1(v^*)\cup\{v^*\}$ consequently,
$V(G^*)\backslash(R^*_1(v^*)\cup\{v^*\})$ is a resolving set for
$G^*$. Thus $\beta(G^*)\leq n(G^*)-|R^*_1(v^*)\cup\{v^*\}|$. On
the other hand, by Propositions~\ref{dim G^*<n-t}, we have
$\beta(G^*)\geq n(G^*)-3$. Therefore,
$|R^*_1(v^*)\cup\{v^*\}|\leq 3$, which yields $|R^*_1(v^*)|\leq
2$. }
\end{proof}
\begin{lemma}\label{R_2=empty}
If $\Gamma_2(v)$ is not homogeneous, then $R_2(v)$ and
$R^*_2(v^*)$ are empty sets.
\end{lemma}
\begin{proof}{
Note that $|R^*_1(v^*)|\geq 1$, otherwise
$\Gamma^*_2(v^*)=\emptyset$. Since $\Gamma_2(v)$ is not
homogeneous, there exist vertices $i,j,$ and $k$ in $\Gamma_2$
such that $i\sim j$ and $j\nsim  k$. Now if
$R_2(v)\neq\emptyset$, then let  $r_1\in R_1(v)\cap \Gamma_1(j)$,
$r_2\in R_2(v)$ and $W_0=\{v,j\}$. Thus, $r(i|W_0)=(2,1)$,
$r(k|W_0)=(2,2)$, $r(r_1|W_0)=(1,1)$ and $r(r_2|W_0)=(1,2)$, and
so $\beta(G)\leq n-4$. This contradiction implies
$R_2(v)=\emptyset$.
\par If  $R^*_2(v^*)\neq\emptyset$, then $\bigcup_{u^*\in R^*_2(v^*)}u^*\nsubseteq
R_2(v)$ and hence, by Lemma~\ref{G^*=K_3}, $v^*$ is of type (N).
Therefore, there exist two adjacent vertices $a,b\in
\bigcup_{u^*\in\Gamma^*_2(v^*)}u^*$, otherwise $\Gamma_2(v)$ is
homogeneous. Since $diam(G^*)=2$, there exist a vertex $r_1\in
\bigcup_{u^*\in R^*_1(v^*)}u^*$ such that $r_1\sim a$. Now let
$v_1,v_2\in v^*,~r_2\in\bigcup_{u^*\in R^*_2(v^*)}u^*$, and
$W=\{v_1,a\}$. Thus, we have $$r(v_2|W)=(2,2),\quad
r(r_1|W)=(1,1),\quad r(r_2|W)=(1,2),\quad r(b|W)=(2,1).$$
Therefore, $V(G)\backslash\{v_2,r_1,r_2,b\}$ is a resolving set
for $G$, with cardinality $n-4$, which contradicts our assumption
$\beta(G)=n-3$. Consequently $R^*_2(v^*)=\emptyset$.}
\end{proof}
\begin{fact}\label{Gamma^*_2<3}
$|\Gamma^*_2(v^*)|\leq 3$.
\end{fact}
\begin{proof}{ If $\Gamma_2(v)$ is homogeneous, then  since every vertex
of $\Gamma^*_2(v^*)$ has a neighbor in $R^*_1(v^*)$ and
$\Gamma^*_2(v^*)$ is homogeneous, for distinct vertices
$x^*,y^*\in\Gamma^*_2(v^*)$, the sets $\Gamma^*_1(x^*)\bigcap
R^*_1(v^*)$ and $\Gamma^*_1(y^*)\bigcap R^*_1(v^*)$ are distinct
nonempty sets. Therefore, for each pair of vertices
$x^*,y^*\in\Gamma^*_2(v^*)$ there exist a vertex $r^*_1\in
R^*_1(v^*)$ such that $r^*_1$ resolves $x^*$ and $y^*$ in $G^*$.
Hence, $R^*_1(v^*)$ resolves all vertices of the set
$\Gamma^*_2(v^*)$. This implies that
$V(G^*)\backslash\Gamma^*_2(v^*)$ is a resolving set for $G^*$,
which yields $\beta(G^*)\leq n(G^*)-|\Gamma^*_2(v^*)|$.  On the
other hand, by Propositions~\ref{dim G^*<n-t}, we have
$\beta(G^*)\geq n(G^*)-3$. Thus, $|\Gamma^*_2(v^*)|\leq 3$.
\par Now let $\Gamma_2(v)$ is not
homogeneous. By Fact~\ref{R_1^*<2}, $|R_1^*(v^*)|\leq 2$. Thus, we
consider two cases, $|R_1^*(v^*)|=1$ and $|R_1^*(v^*)|=2$. In the
case $|R_1^*(v^*)|=1$, let $R^*_1(v^*)=\{r^*_1\}$, $r_1\in r^*_1$,
and for each $l\in\Gamma_2(v)$, $N_1(l)=\Gamma_1(l)\cap
\Gamma_2(v)$, and $N_2(l)=\Gamma_2(l)\cap \Gamma_2(v)$. Since
$\Gamma_2(v)$ is not homogeneous, there exists a vertex
$x\in\Gamma_2(v)$ such that both $N_1(x)$ and $N_2(x)$ are
nonempty sets. Let $a\in N_1(x)$ and $y\in N_2(x)$. Note that
$\Gamma^*_2(v^*)=\{x^*\}\cup N^*_1(x)\cup N^*_2(x)$, and $x$
resolves $a$ and $y$. Hence, if $N_1(x)$ is not homogeneous, then
there exist vertices $i,j,k\in N_1(x)$ such that $i\sim j$ and
$j\nsim  k$. Thus, $W^\prime=\{v,x,j\}$ resolves the
$\{i,k,y,r_1\}$, this contradiction yields $N_1(x)$ is
homogeneous. By a similar reason $N_2(x)$ is also homogeneous.
Note that all different neighbors of vertices in $N^*_1(x)$ are
in the set $N^*_2(x)$, because $N_1(x)$ is homogeneous and its
vertices share their neighbors  in $\Gamma_1(v)\cup\{x\}$.
Similarly every different neighbor of vertices in $N^*_2(x)$ are
in $N^*_1(x)$, hence,  $N^*_1(x)$ and $N^*_2(x)$ resolves each
other. Now let $W_1=N^*_2(x)\cup\{x^*\}$. If each vertex of
$N^*_1(x)$ has a none-neighbor vertex in $N^*_2(x)$, then the
representation of each vertex of $N^*_1(x)$ with respect to
$N^*_2(x)$ has a coordinate $2$, all coordinates of
$r(r^*_1|N^*_2(x))$ is $1$, and $r(v^*|N^*_2(x))$ is entirely
$2$. Consequently, $W_1$ resolves $N^*_1(x)\cup\{r^*_1,v^*\}$.
Thus, $\beta(G)=n-3$ implies that $|N^*_1(x)|\leq 1$. Moreover,
if there exist a vertex $a^*\in N^*_1(x)$ such that $a^*$ is
adjacent to all vertices of $N^*_2(x)$, then $N^*_1(x)$ has at
most two vertices. Otherwise, there are two distinct vertices
$b^*,c^*\in N^*_1(x)$ such that they are different from $a^*$,
and $r(b^*|N^*_2(x))$ and $r(c^*|N^*_2(x))$ are not entirely $1$,
and so $W_1$ resolves four vertices $a^*,~b^*,~c^*$, and $v^*$,
contrary to $\beta(G)=n-3$. Hence, $|N^*_1(x)|\leq 2$.
Furthermore, $|N^*_1(x)|= 2$ yields that there exist a vertex
$a^*\in N^*_1(x)$ such that $a^*$ is adjacent to all vertices of
$N^*_2(x)$. On a similar way $|N^*_2(x)|\leq 2$, moreover,
$|N^*_2(x)|= 2$ only if there exist a vertex $y^*\in N^*_2(x)$
such that $y^*$ is none-adjacent to all vertices of $N^*_1(x)$.
Thus, at most one of the sets $N^*_1(x)$ and $N^*_2(x)$ can have
two vertices, because it is impossible that there exist a pair of
vertices $a^*\in N^*_1(x),~y^*\in N^*_2(x)$ such that $a^*$ is
adjacent to all vertices of $N^*_2(x)$ and $y^*$ is none-adjacent
to all vertices of $N^*_1(x)$. Consequently
$|\Gamma^*_2(v^*)|\leq 4$. We claim that $|\Gamma^*_2(v^*)|=4$ is
impossible.
\par If $|\Gamma^*_2(v^*)|=4$,
then one of the two blow cases can be happened. \vspace{4mm}\\
1. $|N^*_1(x)|=2$ and $|N^*_2(x)|=1$. Let $N^*_1(x)=\{a^*,b^*\}$,
$N^*_2(x)=\{y^*\}$, $y^*\sim a^*$, $v\in v^*,a\in a^*,b\in
b^*,x\in x^*$ and $y\in y^*$. If $a^*\sim b^*$, then $x^*$ and
$b^*$ are twins, see Figure~\ref{3}(a). Since $x^*\sim b^*$, one
of them is of type (N). Note that $b^*$ is not of type (N),
because $N^*_1(x)$ is homogeneous and $a^*\sim b^*$. Hence, $b^*$
is of type (1K) and $x^*$ is of type (N). Thus, the set
$V(G)\backslash\{r_1,y,a,x\}$ is a resolving set for G of size
$n-4$, which is impossible. Therefore, $a^*\nsim b^*$, this gives
the set $V(G)\backslash\{v,r_1,x,a\}$  is a resolving set for $G$.
Which is a contradiction.\vspace{4mm}\\
2. $|N^*_1(x)|=1$ and $|N^*_2(x)|=2$. Let $N^*_1(x)=\{a^*\}$,
$N^*_2(x)=\{y^*,z^*\}$, $z^*\sim a^*$, $v\in v^*,a\in a^*,x\in
x^*,y\in y^*,z\in z^*$ and $S=\{x,y\}$. Thus, $r(z|S)=(2,*)$,
$r(a|S)=(1,2)$, $r(r_1|S)=(1,1)$ and $r(v|S)=(2,2)$, where $*=1$
or $2$. Note that $\beta(G)=n-3$ yields $*=2$, see
Figure~\ref{3}(b). Therefore, $x^*$ and $z^*$ are twins. Since
they are none-adjacent vertices, one of them is of type (K).
Clearly $z^*$ is not of type (K), otherwise, since $N^*_2(x)$ is
homogeneous, we have $*=1$, which is impossible. Consequently,
$x^*$ is of type (K), this implies that the set
$(x^*\backslash\{x\})\cup\{z,y\}$
resolves $\{v,r_1,a,x\}$. Thus, $\beta(G)\leq n-4$. \vspace{4mm}\\
These contradictions yield, $|\Gamma^*_2(v^*)|\leq 3$, when
$|R^*_1(v^*)|=1$.
\par To complete
the proof we need only to consider the case $|R^*_1(v^*)|=2$. In
this case, since all different neighbors of vertices in $R_1(v)$
are in $\Gamma_2(v)$, $|R^*_1(v^*)|=2$ implies that there exist a
vertex $a\in\Gamma_2(v)$ such that $\Gamma_2(a)\cap
R_1(v)\neq\emptyset$. Let $T=\Gamma_2(v)\backslash\{a\}$. Since
$\Gamma_2(v)$ is not homogeneous, it has at least three vertices,
hence $|T|\geq 2$. If $T$ is not homogeneous, then there exists a
vertex $i\in T$ such that $i$ resolves two vertices of $T$.
Moreover,  we know that $a$ resolves two vertices of $R_1(v)$.
Hence,  $\{v,i,a\}$ resolves at least four vertices of $G$. Thus,
we obtain a resolving set for $G$ of size $n-4$, which is
impossible. Therefore, $T$ is homogeneous.
\par If $a$ is adjacent to a vertex in $T$, then $a$
is adjacent to all vertices of $T$, otherwise $\{a,v\}$ resolves
four vertices. If a vertex $t\in T$ is not adjacent to some
vertices of $R_1(v)$, then similar to above, it can be seen that
$\Gamma_2(v)\backslash\{t\}$ is homogeneous and $t$ is adjacent or
none-adjacent to all vertices of $\Gamma_2(v)\backslash\{t\}$.
This implies that $\Gamma_2(v)$ is homogeneous, which is a
contradiction with the assumption. Hence, every vertex of $T$ is
adjacent to all vertices of $R_1(v)$. Therefore, all vertices of
$T$ are twins and form a vertex $b^*$ in $G^*$. Thus,
$\Gamma^*_2(v^*)$ consists of two vertices $a^*$ and $b^*$, where
$a^*$ is of type (1) and $b^*$ is of type (KN). }
\end{proof}
\vspace{-.7cm}\begin{figure}[ht]
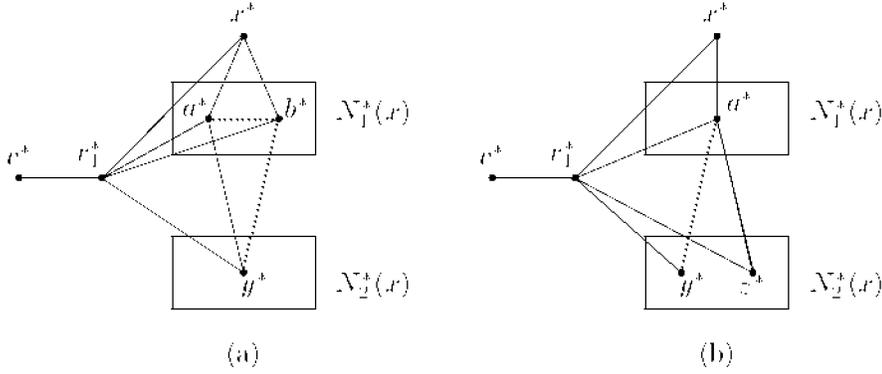
\hspace{1cm}
\vspace*{5cm}\special{em:graph fig2.bmp} \vspace*{1.5cm}
\vspace{-1.3cm}
\caption{\label{3} \footnotesize
Different cases of $N^*_1(x)$ and $N^*_2(x)$.}
\end{figure}
\begin{pro}\label{case2R^*=1} If $|R^*_1(v^*)|=1$, then $G^*$ satisfies one of the
structures $G_2$, $G_3$, or $G_7$.
\end{pro}
\begin{proof}{
Let $R^*_1(v^*)=\{r^*_1\}$. First let $\Gamma_2(v)$ is
homogeneous. Therefore, every vertex of $\Gamma^*_2(v^*)$ is
adjacent to $r^*_1$ and all vertices of $\Gamma^*_2(v^*)$ are
twins. Consequently, since $\Gamma_2(v)$ is homogeneous, all
vertices of $\bigcup_{u^*\in\Gamma^*_2(v^*)}u^*$ are twins. This
gives $|\Gamma^*_2(v^*)|=1$. Now, If $R^*_2(v^*)=\emptyset$, then
$G^*\cong P_3$ and $\alpha(G^*)\geq 2$, otherwise $\beta(G)=n-2$.
It is easy to check that in both cases $\alpha(G^*)=2$ and $3$, at
least one of the leaves is of type (K), otherwise $\beta(G)=n-2$.
Let $G^*=(x^*_1,x^*_2,x^*_3)$ and $x^*_1$ be of type (K). If
$x^*_2$ is of type (1K) and $x^*_3$ is of type (1), then
$\beta(G)=n-2$. This contradiction implies that, if $x^*_3$ is of
type (1), then $x^*_2$ is of type (N). That implies $G^*$
satisfies the structure $G_2$.
\par If $R^*_2(v^*)\neq\emptyset$,
then all vertices of $R^*_2(v^*)$ have a neighbor in $R^*_1(v^*)$,
otherwise $diam(G^*)\geq 3$. Since $\Gamma_1(v)$ is homogeneous,
every vertex of $R^*_2(v^*)$ is adjacent to every vertex of
$R^*_1(v^*)$. Hence, all vertices of $R^*_2(v^*)$ are twins. This
implies that $\Gamma_1(v)$ is a clique, all vertices of
$\bigcup_{u^*\in R^*_2(v^*)}u^*$ are twins and so
$|R^*_2(v^*)|=1$. In this case $G^*$ is the paw and one of the
degree-$2$ vertices is $v^*$ and the other belongs to
$\Gamma^*_1(v^*)$. Therefore, the structure of $G^*$ is as
Figure~\ref{(3),(4)}(a). Hence, $x^*$ and $v^*$ are adjacent
twins. Therefore, one of them is of type (N), otherwise $x^*=v^*$,
which contradicts $R^*_2(v^*)\neq\emptyset$. Since $\Gamma_1(v)$
is a clique, $x^*,y^*$ are of type (1K). Thus, $v^*$ is of type
(N). If $z^*$ is of type (K), then we choose arbitrary fixed
vertices $x\in x^*,v_1,v_2\in v^*,y\in y^*$ and $z_1,z_2\in z^*$.
Hence, $$r(v_1|\{v_2,z_2\})=(2,2),\, r(x|\{v_2,z_2\})=(1,2),\,
r(y|\{v_2,z_2\})=(1,1),\, r(z_1|\{v_2,z_2\})=(2,1).$$ Thus, the
set $V(G)\backslash\{x,y,z_1,v_1\}$ is a resolving set for $G$,
and $\beta(G)\leq n-4$. Consequently, $z^*$ is of type (1N).
Similarly if $x^*$ is of type (K) and $z^*$ is of type (N), then
$V(G)\backslash\{x,y,z_1,v_1\}$ is a resolving set for $G$ that is
a contradiction. Hence, $G^*$ satisfies the structure $G_3$.
\par Now let $\Gamma_2(v)$ is not homogeneous.
By the same notation as the proof of Fact~\ref{Gamma^*_2<3}, we
can see that the vertices of $N_1(x)$ can only be twins with each
other and $x$, also the vertices of $N_2(x)$ can only be twins
with each other, $x$, and $v$. By Lemma~\ref{R_2=empty}, we have
$R^*_2(v^*)=\emptyset$. Hence, if $|\Gamma^*_2(v^*)|=1$, then
$G^*\cong P_3$, all vertices of $N_2(x)$ are twins with $v$, and
all vertices of $N_1(x)$ are twins with $x$. Thus, $x^*$ is of
type (K), $v^*$ is of type (N), and $r^*_1$ is of any type. In
this case  $G^*$ satisfies the structure $G_2$.
\par When $|\Gamma^*_2(v^*)|=2$, there exist three cases.\vspace{4mm}\\
1. The vertex $x$ and all vertices of $N_1(x)$ are twins and there
exist some vertices in $N_2(x)$ which are not twins with $v$ and
they are twins by themselves. Thus, the vertices of
$N_2(x)\backslash v^*$ form exactly one vertex, $y^*$ in $G^*$.
Hence, $x^*$ is of type (K) and $x^*\nsim  y^*$. Therefore,
$y^*$ and $v^*$ are twins. Since $v^*\nsim y^*$, one of them is
of type (K). Note that, if $v^*$ is of type (K), then there exist
some vertex $u\in V(G)$ which is adjacent to all vertices of
$\Gamma_1(v)\cup\{v\}$ and is not adjacent to any vertex of
$\Gamma_2(v)$. Hence, $u\in R_2(v)$, which is impossible, because
by Lemma~\ref{R_2=empty}, $R_2(v)=\emptyset$. Thus, $y^*$ is of
type (K). Let $r_1\in r^*_1$, $y\in y^*$, $x\in x^*$ and $v\in
v^*$. Then $(x^*\backslash\{x\})\cup(y^*\backslash\{y\})$ resolves
$\{v,r_1,x,y\}$, this contradicts our assumption $\beta(G)=n-3$.
Consequently, this case does not happen.
\vspace{4mm}\\
2. There exist some vertices of $N_2(x)$ which are twins with $x$
and the rest are twins with $v$, and all vertices of $N_1(x)$ are
twins. Therefore,  the vertices of $N_1(x)$  create a vertex
$a^*$ in $G^*$, $a^*\sim x^*$, and $x^*$ is of type (N). Hence,
$G^*$ is the paw, with the leaf $v^*$, the degree-$3$ vertex
$r^*_1$, and degree-$2$ vertices $a^*$ and $x^*$. If $a^*$ is of
type (N), then $V(G)\backslash\{v,r_1,x,a\}$ is a resolving set
for $G$, where $r_1\in r^*_1,\- v\in v^*,\-x\in x^*$ and $a\in
a^*$. This contradiction yields $a^*$ is of type (1K). Because
$R_2(v)=\emptyset$, $v^*$ is not of type (K). Therefore,  $v^*$ is
of type (1N). By a similar method as before, we deduce that, if
$a^*$ is of type (K), then $v^*$ and $r^*_1$ can not be of type
(N). Thus, $G^*$ satisfies in structure $G_3$.
\vspace{4mm}\\
3. All vertices of $N_2(x)$ are twins with $v$, and there exist
some vertices in $N_1(x)$ which are not twins with $x$. Hence, the
vertices of $N_1(x)\backslash x^*$ form a unique vertex $a^*$ in
$G^*$, $a^*\sim x^*$, and consequently $G^*$ is the paw. Note that
$v^*$ is the leaf and its type is (N), the vertex $r^*_1$ has
degree $3$, and $x^*,\-a^*$ are degree-$2$ vertices. Since $x^*$
and $a^*$ are adjacent twins, one of them is of type (N). Also,
since all vertices of $N^*_2(x)$ are twins with $v$, $x^*$ is of
type (1K), and so $a^*$ is of type (N). Therefore, $G^*$ has the
structure $G_3$.
\par Finally, if $|\Gamma^*_2(v^*)|=3$, then we have three
following cases.\vspace{4mm}\\
1. Every vertex of $N_2(x)$ is twin with $v$ or $x$, and
$N^*_1(x)$ has two vertices $a^*$ and $b^*$. In this case $a^*$
and $b^*$ are twins, hence $a^*=b^*$, because $N_1(x)$ is
homogeneous. Thus, $|\Gamma^*_2(v^*)|\leq 2$. Therefore this is
not
the case.\vspace{4mm}\\
2. Every vertex of $N_1(x)$ is twin with $x$, and $N^*_2(x)$ has
two vertices $y^*$ and $z^*$. Hence,  $y^*$ and $z^*$ are twins.
Since $N_2(x)$ is homogeneous, $y^*=z^*$, which is a contradiction
with $|\Gamma^*_2(v^*)|=3$.\vspace{4mm}\\
3. There exist some vertices in $N_1(x)$ which are not twins with
$x$, also there exist some vertices in $N_2(x)$ which are neither
twins with $x$ nor $v$. In this case, each one of $N^*_1(x)$ and
$N^*_2(x)$ has exactly one vertex $a^*$ and $y^*$, respectively.
If $a^*\nsim y^*$, then $v^*$ and $y^*$ are none-adjacent twins.
Hence one of them is of type (K). Since $R_2(v)=\emptyset$, $v^*$
is of type (1N), and so $y^*$ is of type (K). Let $v\in v^*$,
$y\in y^*$, $a\in a^*$, and $x\in x^*$. Then
$V(G)\backslash\{y,x,r_1,v\}$ is a resolving set for $G$. This
contradiction yields $a^*\sim y^*$, and in consequence, $G^*$ is a
kite with a pendant edge, adjacent to a degree-$3$ vertex. Thus,
$y^*$ and $x^*$ are none-adjacent twins, hence one of them is of
type (K). By symmetry, we can assume  $x^*$ is of type (K). Since
$R_2(v)=\emptyset$, $v^*$ is of type (1N). As we see before,
$v^*,\-a^*$, and $r_1^*$ are not of type (N) and $y^*$ is not of
type (KN). Therefore, $G^*$ satisfies the structure $G_7$ and the
proof is completed.}\end{proof}
\vspace{-.7cm}\begin{figure}[ht]
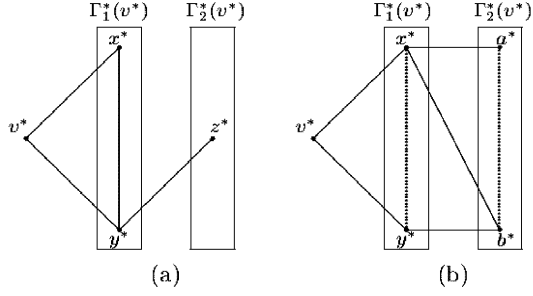
\hspace{3cm}
\vspace*{5cm}\special{em:graph fig3.bmp} \vspace*{1.5cm}
\vspace{-2.3cm}
\caption{\label{(3),(4)}
{\footnotesize$|\Gamma^*_2(v^*)|=1$ and $|\Gamma^*_2(v^*)|=2$.}}
\end{figure}
 Heretofore we have considered the case
$|R^*_1(v^*)|=1$. Now we investigate the case $|R^*_1(v^*)|=2$.
\begin{pro}\label{case2R^*=2} If $|R^*_1(v^*)|=2$
and $\Gamma_2(v)$ is not homogeneous, then $G^*$ satisfies the
structure $G_7$.
\end{pro}
\begin{proof}{  By the same notation as the proof of
Fact~\ref{Gamma^*_2<3}, let $\Gamma^*_2(v^*)=\{a^*,b^*\}$, where
$a^*$ is of type (1) and $b^*$ is of type (KN). Thus, if $a^*\sim
b^*$, then $b^*$ is of type (N), however $a^*\nsim b^*$ implies
that $b^*$ is of type (K), because $\Gamma_2(v)$ is not
homogeneous. Let $R_1^*(v^*)=\{x^*,y^*\}$. Since $R_1(v)$ has a
none-adjacent vertex to $a$, the vertex $a^*$ has exactly one
neighbor in $R^*_1(v^*)$. By the proof of Fact~\ref{Gamma^*_2<3},
$b^*\sim x^*$ and $b^*\sim y^*$, thus $G^*$ is the $4$-cycle,
$C_4=(v^*,x^*,b^*,y^*,v^*)$ with the pendant edge $x^*a^*$ and
possibly extra edges $a^*b^*$ and $x^*y^*$, see
Figure~\ref{(3),(4)}(b). Because $diam(G^*)=2$, at least one of
the edges $a^*b^*$ and $x^*y^*$ exists. If $a^*\sim b^*$, then
$b^*$ is of type (N). Let $v\in v^*$, $a\in a^*$, $b\in b^*$,
$x\in x^*$, and $y\in y^*$. Consequently, the set $a^*\cup
(b^*\backslash\{b\})$ resolves $\{v,x,b,y\}$, since $b^*$ is of
type (N). This contradiction yields $a^*\nsim  b^*$, and so
$x^*\sim y^*$, $a^*$ is of type (1), and $b^*$ is of type (K).
Note that $x^*$ and $y^*$ are not of type (N), otherwise
$\Gamma_1(v)$ is not homogeneous. Also, we can see easily that
$v^*$ is not of type (KN). In this case $G^*$ satisfies the
structure $G_7$.}
\end{proof}
Now, we need only to consider the case $|R^*_1(v^*)|=2$ when
$\Gamma_2(v)$ is homogeneous. In this case, if
$|\Gamma^*_2(v^*)|=1$, then all vertices of $R_1(v)$ are twins and
consequently, $|R^*_1(v)|=1$, which contradicts $|R^*_1(v^*)|=2$.
Therefore, $|\Gamma^*_2(v^*)|\geq 2$.
\begin{lemma}\label{Claim5}
If $\Gamma_2(v)$ is homogeneous and $|R^*_1(v^*)|=2$, then
$R^*_2(v^*)=\emptyset$.
\end{lemma}
\begin{proof} { By Fact~\ref{Gamma^*_2<3} and above argument, we
have $2\leq |\Gamma^*_2(v^*)|\leq 3$. Suppose on the contrary that
$R^*_2(v^*)\neq\emptyset$. Since $diam(G^*)=2$ and $\Gamma_1(v)$
is homogeneous, every vertex of $R^*_2(v^*)$ is adjacent to every
vertex of $R^*_1(v^*)$. In this way all vertices of $R^*_2(v^*)$
are twins, therefore all vertices of $\bigcup_{u^*\in
R^*_2(v^*)}u^*$ are twins, and so $|R^*_2(v^*)|=1$. Let
$R^*_2(v^*)=\{r^*_2\}$, $R^*_1(v^*)=\{x^*,y^*\}$, and
$\{a^*,b^*\}\subseteq\Gamma^*_2(v^*)$. Therefore, $r^*_2$ is
adjacent to the vertices $v^*,x^*$ and $y^*$. Note that
$\Gamma_1(v)$ and $\Gamma^*_1(v^*)$ are cliques, because
$\Gamma_1(v)$ is homogeneous and $r^*_2$ is adjacent to $x^*$ and
$y^*$. Thus, $r^*_2,x^*$ and $y^*$ are of type (1K). Since all
neighbors of $r^*_2$ and $v^*$ are shared, $r^*_2$ and $v^*$ are
adjacent twins. For $r^*_2\neq v^*$, one of them is of type (N).
Because $r^*_2$ is of type (1K), $v^*$ is of type (N).
We choose arbitrary fixed vertices $v_1,v_2\in v^*,r_2\in
r^*_2,x\in x^*,y\in y^*,a\in a^*$ and $b\in b^*$ and set
$T=\{v_1,a,b\}$. Since $a^*\neq b^*$ and $\Gamma_2(v)$ is
homogeneous, one of the vertices $a^*$ and $b^*$has only one
neighbor in $R^*_1(v^*)$ and the other have one or two. Without
loss of generality, we can assume that $x^*$ is the only neighbor
of $a^*$ in $R^*_1(v^*)$, see Figure~\ref{2}(b). Thus,
 $$r(v_2|T)=(2,2,2),\quad r(x|T)=(1,1,*),\quad
r(r_2|T)=(1,2,2),\quad r(y|T)=(1,2,1),$$ where $*$ is $1$ or $2$.
However, the four above representations are distinct. This yields
the contradiction $\beta(G)\leq n-4$. Therefore,
$R^*_2(v^*)=\emptyset$. }\end{proof}
\begin{lemma}\label{claim3}
If $\Gamma_2(v)$ is homogeneous and $|\Gamma^*_2(v^*)|=3$, then
there is exactly one vertex $a^*\in\Gamma^*_2(v^*)$ such that
$a^*$ is adjacent to all vertices of $R^*_1(v^*)$.
\end{lemma}
\begin{proof} { Let $|\Gamma^*_2(v^*)|=3$. We have seen that $R^*_1(v^*)$
resolves the set $\Gamma^*_2(v^*)$. Suppose on the contrary that
there is not any vertex of $\Gamma^*_2(v^*)$ adjacent to all
vertices of $R^*_1(v^*)$. Hence, at least one coordinate of the
representation of each vertex in $\Gamma^*_2(v^*)$, with respect
to $R^*_1(v^*)$ is 2. While every coordinate of
$r(v^*|R^*_1(v^*))$ is 1. Therefore, $R^*_1(v^*)$ is a resolving
set for $G^*[R^*_1(v^*)\cup\Gamma^*_2(v^*)\cup\{v^*\}]$ with
cardinality $n(G^*)-4$, since $|\Gamma^*_2(v^*)\cup \{v^*\}|=4$.
It follows that $\beta(G^*)\leq n(G^*)-4$, and
Proposition~\ref{dim G^*<n-t} implies that $\beta(G)\leq n-4$,
which is a contradiction. Therefore, there exists a vertex
$a^*\in\Gamma^*_2(v^*)$ adjacent to all vertices of $R^*_1(v^*)$.
If there exists another vertex $b^*\in\Gamma^*_2(v^*)$ adjacent to
all of $R^*_1(v^*)$, then $a^*$ and $b^*$ are twins, since
$\Gamma^*_2(v^*)$ is homogeneous. This implies that $a^*=b^*$
while $|\Gamma^*_2(v^*)|=3$. Therefore, such a vertex in
$\Gamma^*_2(v^*)$ is unique. }
\end{proof}
\begin{lemma}\label{claim4}
If $\Gamma_2(v)$ is homogeneous and $|R^*_1(v^*)|=2$, then
$|\Gamma^*_2(v^*)|\leq 2$.
\end{lemma}
\begin{proof} {On the contrary, suppose
$|\Gamma^*_2(v^*)|=3$.  By Lemma~\ref{claim3}, there exists
exactly one vertex $a^*\in\Gamma^*_2(v^*)$ such that $a^*$ is
adjacent to all vertices of $R^*_1(v^*)$. Let
$R^*_1(v^*)=\{x^*,y^*\}$ and $\Gamma^*_2(v^*)=\{a^*,b^*,c^*\}$.
Each one of vertices $b^*$ and $c^*$ has at least one neighbor in
$R^*_1(v^*)$ and by Lemma~\ref{claim3}, they have exactly one
neighbor in $R^*_1(v^*)$. If their neighbors in $R^*_1(v^*)$ are
same, then they are twins, since $\Gamma^*_2(v^*)$ is homogeneous.
This implies that every pair of vertices $b\in b^*$ and $c\in c^*$
are twins (because $\Gamma_2(v)$ is homogeneous) consequently,
$b^*=c^*$, which contradicts $|\Gamma^*_2(v^*)|=3$. Thus, one of
them, say $b^*$, is (only) adjacent to $x^*$ and the other $c^*$,
is (only) adjacent to $y^*$, see Figure~\ref{2}(a) (dotted edges
may be exist or not).
Now$$r(v^*|\{b^*,c^*\})=(2,2), \, r(x^*|\{b^*,c^*\})=(1,2), \,
r(y^*|\{b^*,c^*\})=(2,1), \, r(a^*|\{b^*,c^*\})=(*,*),$$where $*$
is $1$ or $2$. If $*=1$, then $r(a^*|\{b^*,c^*\})=(1,1)$, and so
$V(G^*)\backslash\{a^*,v^*,x^*,y^*\}$ is a resolving set of size
$n(G^*)-4$ for $G^*$, this contradiction yields $*=2$. Hence
$\Gamma_2(v)$ and $\Gamma^*_2(v^*)$ are independent sets, since
$\Gamma_2(v)$ is homogeneous.
\par Since $R^*_2(v^*)=\emptyset$, if $v^*$ is of type (1N),
then $v^*$ and $a^*$ are twins and every pair of vertices $v\in
v^*$ and $a\in a^*$ are twins (because both $a^*$ and $v^*$ are of
type (1N)), and so $v^*=a^*$, that is a contradiction. Therefore,
$v^*$ is of type (K). For arbitrary fixed vertices $v_1,v_2\in
v^*,x\in x^*,y\in y^*,a\in a^*,b\in b^*$ and $c\in c^*$ and
$T=\{v_1,a,c\}$, we have
$$r(v_2|T)=(1,2,2),\quad r(x|T)=(1,1,2),\quad r(y|T)=(1,1,1),\quad
r(b|T)=(2,2,2).$$ Hence,  $V(G)\backslash\{v_2,x,y,b\}$ is a
resolving set for $G$, with cardinality $n-4$. This contradiction
implies that $|\Gamma^*_2(v^*)|\leq 2$.}
\end{proof}

\vspace{-.7cm}\begin{figure}[ht]
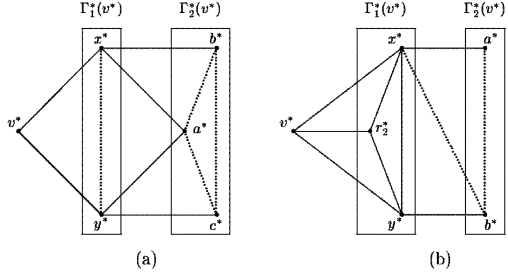
\hspace{3cm}
\vspace*{5cm}\special{em:graph fig4.bmp} \vspace*{1.5cm}
\vspace{-2.3cm}
\caption{\label{2} \footnotesize$|\Gamma^*_2(v^*)|=3$ and
$|\Gamma^*_1(v^*)|=3$.}
\end{figure}

On account of the above results, we need only assume that
$|\Gamma^*_2(v^*)|=|R^*_1(v^*)|=2$ and $|R^*_2(v^*)|\leq~1$.

\begin{pro}
If $|R^*_1(v^*)|=2$ and $\Gamma_2(v)$ is homogeneous, then $G^*$
satisfies one of the structures $G_4$ to $G_7$.
\end{pro}
\begin{proof}{
 Let $R^*_1(v^*)=\{x^*,y^*\}$,
$\Gamma^*_2(v^*)=\{a^*,b^*\}$. Then $G^*$ is as described in
Figure~\ref{4}. If $a^*\nsim b^*$, then $x^*\sim y^*$ and $x^*\sim
b^*$, otherwise $diam(G^*)=3$, a contradiction. Let $G_0$ be the
path $(a^*,x^*,v^*,y^*,b^*)$. Thus, $G^*$ must be one of the
following five graphs:
\begin{description}
\item $H^*_1:=G_0+a^*b^*$, \item $H^*_2:=G_0+a^*b^*+x^*b^*$, \item
$H^*_3:=G_0+a^*b^*+x^*y^*$, \item
$H^*_4:=G_0+a^*b^*+x^*b^*+x^*y^*$, \item
$H^*_5:=G_0+x^*b^*+x^*y^*$.
\end{description}
We fix the vertices $v\in v^*,x\in x^*,y\in y^*,a\in a^*$ and
$b\in b^*$ in each $H^*_i$, $1\leq i\leq 5$. Note that $H^*_1\cong
C_5$. If $G^*\cong H^*_1$, then all vertices of $H^*_1$ are of
(1), otherwise (with a simple computation) $\beta(G)\leq n-4$. In
this case $G^*$ satisfies structure $G_4$.
\par When $G^*\cong
H^*_2$, $x^*$ and $y^*$ are not of type (K), because $\Gamma_1(v)$
is homogeneous and $x^*\nsim y^*$. Similarly $a^*$ and $b^*$ are
not of type (N). If $x^*$ or $y^*$ is of type (N), then
$V(G)\backslash\{v,x,y,b\}$ is a resolving set for $G$, with
cardinality $n-4$. Also $v^*$ of type (N) or (K) yields
$V(G)\backslash\{v,x,y,b\}$ or $V(G)\backslash\{v,x,a,b\}$ is a
resolving set for $G$  of size $n-4$, respectively. These
contradictions show that $G^*$ satisfies the structure $G_5$ if
$G^*\cong H^*_2$.
\par Let $G^*\cong H^*_3$. Since $\Gamma_1(v)$
and $\Gamma_2(v)$ are homogeneous, $x^*\sim y^*$ and $a^*\sim b^*$
imply that $x^*,y^*,a^*,b^*$ are not of type (N). If $a^*$ or
$b^*$ is of type (K), then $V(G)\backslash\{v,x,y,a\}$ or
$V(G)\backslash\{x,y,v,b\}$ is an $(n-4)$-vertex resolving set for
$G$, respectively. Also $v^*$ of type (N) yields the resolving
set, $V(G)\backslash\{x,y,v,b\}$ for $G$. These contradictions
imply that $G^*$ has  the structure $G_5$.
\par If $G^*\cong H^*_4$
and one of the vertices $v^*$ or $y^*$ is of type (N), then the
set $V(G)\backslash\{v,x,y,b\}$ or $V(G)\backslash\{x,y,a,b\}$ is
a resolving set for $G$ of cardinality $n-4$. Thus $v^*$ and $y^*$
are of type (1K). By symmetry, the vertices $a^*$ and $b^*$ are of
type (1K), too.  If non-adjacent vertices $v^*$ and $b^*$ are of
type (K), then the set $V(G)\backslash\{v,x,y,b\}$ is a resolving
set of size $n-4$. Similarly none-adjacent vertices  $a^*$ and
$y^*$ are not of type (K) simultaneously. Also, if none-adjacent
vertices $a^*$ and $v^*$ are of type (K), then
$V(G)\backslash\{v,x,y,a\}$ resolves $G$, which is impossible.
Therefore,  none-adjacent vertices are not of the same type (K).
Moreover, if $x^*$ is of type (N), and $y^*$ or $v^*$ is of type
(K), then $V(G)\backslash\{v,x,y,a\}$ is a resolving set for $G$
of cardinality $n-4$. By the same reason, if $x^*$ is of type (N),
then $a^*$ and $b^*$ are not of type (K). Thus, $G^*$ satisfies
the structure $G_6$.
\par Finally, assume that $G^*\cong H^*_5$. Since $v^*\neq b^*$
and these vertices are none-adjacent twins in $H^*_5$, at least
one of them is of type (K). Hence, $v^*$ is of type (K) and $b^*$
is of type (1N), because $\Gamma_2(v)$ is homogeneous and
$a^*\nsim b^*$. If $b^*$ is of type (N), then
$V(G)\backslash\{v,x,y,b\}$ resolves $G$, a contradiction. It
follows that $b^*$ is of type (1). By the similar way as before,
one can see that $a^*$ is of type (1), and both $x^*$ and $y^*$
are of type (1K), and thus $G^*$ satisfies the structure $G_7$.}
\end{proof}
\begin{figure}[ht]
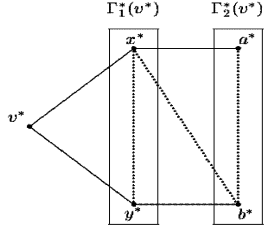
\hspace{4cm}
\vspace*{5cm}\special{em:graph fig5.bmp} \vspace*{1.5cm}
\vspace{-2.8cm}
\caption{\label{4} \footnotesize
$|\Gamma^*_1(v^*)|=|\Gamma^*_2(v^*)|=2$ and both $\Gamma^*_1(v^*)$
and $\Gamma^*_2(v^*)$ are homogeneous.}
\end{figure}


\noindent{\bf Case 2.} For each vertex $v\in V(G)$ with
$\Gamma^*_2(v^*)\neq\emptyset$, $\Gamma_1(v)$ is not homogeneous.
\vspace{4mm}\\
We choose a fix vertex $v\in V(G)$ with the property
$\Gamma^*_2(v^*)\neq\emptyset$. Lemma~\ref{one homogeneous}
concludes that in this case, $\Gamma_2(v)$ is homogeneous. For
each vertex $x\in\Gamma_1(v)$, let
$M_1(x):=\Gamma_1(v)\cap\Gamma_1(x)$ and
$M_2(x):=\Gamma_1(v)\cap\Gamma_2(x)$. Since
$M_2(x)\subseteq\Gamma_2(x)$ and $\Gamma_2(x)$ is homogeneous,
$M_2(x)$ is also homogeneous. If $M_1(x)$ is not homogeneous, then
there exist vertices $i$, $j$, and $k$ in $M_1(x)$ such that
$i\sim j$ and $k\nsim  j$. Thus, for each pair of vertices $y\in
M_2(x)$ and $c\in\Gamma_2(v)$, we have $r(i|\{v,x,j\})=(1,1,1)$,
$r(k|\{v,x,j\})=(1,1,2)$, $r(y|\{v,x,j\})=(1,2,*)$,
$r(c|\{v,x,j\})=(2,*_1,*_2)$, where $*$, $*_1$ and $*_2$ are $1$
or $2$. However, these representations are distinct, which is a
contradiction. Therefore, $M_1(x)$ is homogeneous.
\begin{pro}\label{claim 3.1} If there exist a vertex
$x\in R_2(v)$ with $M_2(x)\neq\emptyset$, then $G^*$ satisfies the
structure $G_6$.
\end{pro}
\begin{proof} { Since $x\in R_2(v)$, we have
$\Gamma_1(x)=M_1(x)\cup\{v\}$. Note that $v$ is adjacent to all
vertices of $M_1(x)$. Since $M_1(x)$ is homogeneous and
$\Gamma_1(x)$ is not homogeneous, we conclude $M_1(x)$ is an
independent set and contains at least two vertices. Now let $m_1$
and $m_2$ be two arbitrary vertices in $M_1(x)$. Thus, $m_1$
resolves $m_2$ and $v$, hence $m_1$ can not resolve any pair of
vertices in $\Gamma_2(x)$, otherwise the set $\{x,m_1\}$ resolves
at least four vertices. Therefore, $m_1$ is adjacent to all
vertices of $\Gamma_2(x)$ or none-adjacent to all of them. Since
$m_1$ is an arbitrary vertex of $M_1(x)$, all vertices of
$\Gamma_2(x)$ have same neighbors in $M_1(x)$. Note that
$\Gamma_2(x)=M_2(x)\cup\Gamma_2(v)$, because $x\in R_2(v)$. Also
all vertices of $M_2(x)$ are adjacent to $v$, and all vertices of
$\Gamma_2(v)$ are not adjacent to $v$. Thus, every pair of
vertices in $M_2(x)$ and also every pair of  vertices of
$\Gamma_2(v)$ are twins. Let $t^*=M_2(x)$ and $s^*=\Gamma_2(v)$ be
the corresponding vertices in $G^*$. Moreover,  the vertices of
$M_1(x)$ that are adjacent to all of $\Gamma_2(x)$ are twins and
form a vertex $y^*$ in $G^*$, also the remaining vertices of
$M_1(x)$ are twins with each other and create a vertex $z^*$ in
$G^*$. Therefore, $G^*$ has at most six vertices $v^*$, $x^*$,
$y^*$, $z^*$, $t^*$, and $s^*$, where $x^*$ is adjacent to $v^*$,
$z^*$, and $y^*$. Also, $v^*$ and $y^*$ are adjacent to all
vertices except $s^*$ and $z^*$, respectively. There is no other
edge in $G^*$ except possibly $s^*t^*$, see Figure~\ref{5}(a).
\par If all
of these six vertices exist, then $d(z^*,s^*)=3$, which
contradicts $diam(G^*)=2$. Since $s^*=\Gamma_2(v)$, the vertex
$z^*$ does not exist. It is clear that $y^*$ is of type (N),
because $M_1(x)$ is an independent set of size at least two. Let
$v\in v^*$, $x\in x^*$, $y_1,y_2\in y^*$, $s\in s^*$, and $t\in
t^*$. If $s^*\nsim t^*$, then $v^*\cup x^*\cup t^*\subseteq
\Gamma_2(s)$. But the set $v^*\cup x^*\cup t^*$ is not
homogeneous, and so $\Gamma_2(s)$ is not homogeneous, this
contradiction yields $s^*\sim t^*$. Thus, since $s^*\cup
t^*\subseteq \Gamma_2(x)$, the adjacent vertices $s^*$ and $t^*$
are of type (1K). Moreover, $x^*\cup v^*\subseteq \Gamma_2(s)$
and $x^*\sim v^*$ yields $x^*$ and $v^*$ are of type (1K). But as
we see before, when $y^*$ is of type (N) the other vertices can
not be of type (K), hence they must be of type (1), also, two
none-adjacent vertices are not of the same type (K). Therefore,
$G^*$ satisfies the structure $G_6$. }
\end{proof}
Now let for each vertex $u\in R_2(v)$, one of the sets $M_1(u)$ or
$M_2(u)$ is empty. Note that every vertex of $R_2(v)$ has a
neighbor in $R_1(v)$, otherwise $diam(G)\geq 3$. Hence,
$M_1(u)\neq\emptyset$. Consequently, $M_2(u)=\emptyset$, for each
$u\in R_2(v)$. Since $diam(G)=2$,
$M_1(u)=\Gamma_1(v)\backslash\{u\}$. Therefore, every vertex of
$R_2(v)$ is adjacent to all vertices of $R_1(v)$, $R_2(v)$ is a
clique, and $|R^*_2(v^*)|\leq 1$. We consider the cases $R_1(v)$
is homogeneous or not homogeneous, separately.
\begin{figure}[ht]\hspace{4cm}
\vspace*{5cm}\special{em:graph fig6.bmp} \vspace*{1.5cm}
\vspace{-3.2cm}
\caption{\label{5} \footnotesize
$|\Gamma^*_1(v^*)|=4$ and $|\Gamma^*_1(v^*)|=2$ in Case 2.}
\end{figure}
\begin{pro}\label{claim32}
 If for each vertex $u\in R_2(v)$, the set $M_2(u)=\emptyset$ and $R_1(v)$ is homogeneous,
 then $G^*$ satisfies the structure $G_2$.
\end{pro}
\begin{proof} { Let $H=G[\{v\}\cup R_1(v)\cup\Gamma_2(v)]$. It is
clear that $H$ is an induced subgraph of $G$ with diameter two. By
Corollary~\ref{sugraph d=2}, $n(H)-3\leq\beta(H)\leq n(H)-2$.
First suppose $\beta(H)=n(H)-2$. Theorem~\ref{n-2} yields $H$ is
one of the graphs $K_{r,s}$, $K_r\vee \overline K_s$, or
$K_r\vee(K_1\cup K_s)$. If $H=K_{r,s}$, then $R_1(v)$ and
$\Gamma_2(v)$ are independent sets. Now, for each $t\in
\Gamma_2(v)$ the set $\Gamma_1(t)$ is a nonempty independent set
in $G$, which is a contradiction with the assumption that
$\Gamma_1(t)$ is not homogeneous, for such a $t$. Consequently,
$H\neq K_{r,s}$. Note that $R_2(v)$ is a clique and all vertices
of $R_2(v)$ are adjacent to all vertices of $R_1(v)$, since for
each vertex $u\in R_2(v)$ the set $M_2(u)$ is empty. Therefore,
$R_1(v)$ is not a clique, otherwise $\Gamma_1(v)$ is homogeneous.
On the other way, if $H=K_r\vee \overline K_s$ or
$H=K_r\vee(K_1\cup K_s)$, then $R_1(v)$ is a clique. This
contradiction yields $\beta(H)=n(H)-3$.
\par Since $R_1(v)$ and $\Gamma_2(v)$ are homogeneous, the graph $H$
with its vertex $v$ satisfies the conditions of Case 1. Thus,
$H^*$ has one of the six structures $G_2$ to $G_7$. If
$R_2(v)=\emptyset$, then $\Gamma_1(v)=R_1(v)$ is homogeneous,
which is a contradiction. Therefore,  $R_2(v)\neq\emptyset$.
Hence, we have $R_2(v)$ is a clique and all its vertices are
adjacent to all vertices of $R_1(v)$, consequently, $R^*_1(v^*)$
is not a clique, otherwise $\Gamma_1(v)$ is homogeneous. Note
that, in the structures $G_3$, $G_6$, and $G_7$, $R^*_1(v^*)$ is
a clique, therefore these structures do not occur.
\par In graphs with
structure $G_5$, when both neighbors of $v^*$ have degree $3$,
$R^*_1(v^*)$ is a clique. This implies that, if $H^*$ has the
structure $G_5$, then $v^*$ is adjacent to a degree-$2$ vertex and
a degree-$3$ vertex. Here, if $t^*\in \Gamma^*_2(v^*)$ is the
degree-$2$ vertex and $t\in t^*$, then $\Gamma_1(t)$ is a clique,
contradicts the assumption that $\Gamma_1(v)$ is not homogeneous.
Therefore, $H^*$ has not the structure $G_5$.
\par When $H^*$ has the structure $G_4$, $H^*\cong C_5$. If $a^*$ is a neighbor of $v^*$, then
$R_2(v)\cup a^*$ is a resolving set for $G$ with cardinality
$n-4$. These contradictions imply that $H^*$ can only have the
structure $G_2$.
\par Let $H^*$ satisfies the structure $G_2$.
If $H^*$ has not the condition (a) of the structure $G_2$, then
the only vertex of $R^*_1(v^*)$ is of type (1K), therefore,
$R_1(v)$ is a clique or $|R_1(v)|=1$, hence $\Gamma_1(v)$ is a
clique, which is impossible. Thus, $H^*$ has the condition (a) of
the structure $G_2$, consequently, the degree-$2$ vertex of $H^*$
is of type (N), hence, $R_1(v)$ is an independent set of size at
least $2$. If $v^*$ as a vertex of $H^*$ is of type (N), then by
condition (a), the other leaf of $H^*$ is of type (K) and
consequently, $V(G)\backslash\{v,r_1,r_2,b\}$ resolves $G$, where
$v\in v^*,~r_1\in R_1(v),~r_2\in R_2(v)$, and $b\in \Gamma_2(v)$.
Therefore,  $v^*$ as a vertex of $H^*$ is of type (1K). Hence,
since all vertices of $R_2(v)$ are adjacent to all vertices of
$\Gamma_1(v)\cup\{v\}$, $v$ is twin with all vertices of $R_2(v)$.
Consequently, $v^*$ as a vertex of $G^*$ is of type (1K) and $G^*$
satisfies the structure $G_2$. }
\end{proof}
We investigate the case $R_1(v)$ is not homogeneous for two
possibility, $|\Gamma^*_2(v^*)|\geq 2$ and $|\Gamma^*_2(v^*)|=~1$,
separately.
\begin{pro}\label{claim3-3}
If for each $u\in R_2(v)$ the set $M_2(u)=\emptyset$, $R_1(v)$ is
not homogeneous, and $|\Gamma^*_2(v^*)|\geq 2$, then $G^*$
satisfies the structure $G_6$.
\end{pro}
\begin{proof}{ Since $\Gamma_2(v)$ is homogeneous,
$V(G^*)\backslash\Gamma^*_2(v^*)$ is a resolving set for $G^*$.
Hence, Proposition~\ref{dim G^*<n-t} implies that
$\Gamma^*_2(v^*)$ has at most three vertices. Note that, there
exist a vertex $x\in R_1(v)$ such that both sets $M_1(x)\cap
R_1(v)$ and $M_2(x)\cap R_1(v)$ are nonempty sets, because
$R_1(v)$ is not homogeneous. We have $M_1(x)\cap R_1(v)$ and
$M_2(x)\cap R_1(v)$ are homogeneous, because $M_1(x)$ and $M_2(x)$
are homogeneous. On the other hand, all vertices of $R_2(v)$ are
adjacent to all vertices of $R_1(v)$, since for each vertex $u\in
R_2(v)$ the set $M_2(u)$ is empty. Also, $(M^*_2(x)\cap
R^*_1(v^*))\cup\Gamma^*_2(v^*)$ resolves $M^*_1(x)\cap
R^*_1(v^*)$, because all different neighbors of the set
$M^*_1(x)\cap R^*_1(v^*)$ are in the set $(M^*_2(x)\cap
R^*_1(v^*))\cup\Gamma^*_2(v^*)$. Moreover, in the representations
of all vertices of $M^*_1(x)\cap R^*_1(v^*)$ with respect to
$\Gamma^*_2(v^*)$, at least one of the coordinates is $1$. While
all coordinates of $r(v^*|\Gamma^*_2(v^*)$ are $2$. Hence,
$(M^*_2(x)\cap R^*_1(v^*))\cup\Gamma^*_2(v^*)$ resolves
$\{v^*\}\cup(M^*_1(x)\cap R^*_1(v^*))$. Now Proposition~\ref{dim
G^*<n-t} gives $|\{v^*\}\cup(M^*_1(x)\cap R^*_1(v^*))|\leq 3$, and
consequently $|M^*_1(x)\cap R^*_1(v^*)|\leq 2$. Similarly
$|M^*_2(x)\cap R^*_1(v^*)|\leq 2$.
\par Since all neighbors of $\Gamma^*_2(v^*)$ are
in $R^*_1(v^*)$ and $\Gamma_2(v)$ is homogeneous,
$|\Gamma^*_2(v^*)|\geq 2$ implies that there exist vertices
$z^*\in R^*_1(v^*)$ and $t^*\in\Gamma^*_2(v^*)$ such that
$z^*\nsim t^*$. Let $z\in z^*$. Then $z\nsim t$, for each $t\in
t^*$. Since $z$ has  a neighbor $t^\prime\in\Gamma_2(v)$, $z$ is
adjacent to all vertices of $R_1(v)\backslash\{z\}$ or not
adjacent to all these vertices, otherwise the set $\{v,z\}$
resolves four vertices of $G$. Moreover, if
$R_1(v)\backslash\{z\}$ is not homogeneous, then there exist three
vertices $i,j,k\in R_1(v)\backslash\{z\}$ such that $j$ resolves
$\{i,k\}$, and so $\{v,z,j\}$ resolves $\{i,k,t,t^\prime\}$, which
is impossible. Thus, $R_1(v)\backslash\{z\}$ is homogeneous.
Therefore, $G[R_1(v)]\cong K_1\vee\overline{K_l}$ or $K_1\cup K_l$
for some positive integer $l\geq 2$, because $R_1(v)$ is not
homogeneous. It follows that  all vertices of
$R_1(v)\backslash\{z\}$ have a neighbor and a none-neighbor vertex
in $R_1(v)$. Hence, each vertex of $R_1(v)\backslash\{z\}$ is
adjacent or none-adjacent to all vertices of $\Gamma_2(v)$, since
$\beta(G)=n-3$. But by definition of $R_1(v)$, each vertex of
$R_1(v)$ has a neighbor in $\Gamma_2(v)$. From this, each vertex
of $R_1(v)\backslash\{z\}$ is adjacent to all vertices of
$\Gamma_2(v)$. Thus, all vertices of $R_1(v)\backslash\{z\}$ are
twins, and consequently they form a vertex $y^*$ of type (KN) in
$G^*$, and $z^*$ is a vertex of type (1) in $G^*$. Therefore,
$|R^*_1(v^*)|=2$ and $y^*$ is adjacent to all vertices of
$\Gamma^*_2(v^*)$.
\par If $|\Gamma^*_2(v^*)|=3$, then $z^*$ can
adjacent to one or two vertices of $\Gamma^*_2(v^*)$, however two
vertices of $\Gamma^*_2(v^*)$ coincide, because $\Gamma_2(v)$ is
homogeneous. Hence, $|\Gamma^*_2(v^*)|=2$. Let
$\Gamma^*_2(v^*)=\{r^*,s^*\}$. Then $G^*[V(G^*)\backslash
R^*_2(v^*)]$ is as illustrated in Figure~\ref{5}(b). If both edges
$y^*z^*$ and $r^*s^*$ do not exist, then $d(r^*,z^*)=3$, which
contradicts $diam(G^*)=2$, therefore one of them exists. Let $y\in
y^*,\-r\in r^*,\-s\in s^*$. If $y^*\sim z^*$ and $r^*\nsim s^*$,
then $y^*$ is of type (N), since $\Gamma_1(v)$ is not homogeneous.
Also, $r^*$ is of type (K), otherwise $\Gamma_1(r)=y^*$ is an
independent set, which is impossible. This this leads to the
$(n-4)$-vertex resolving set, $V(G)\backslash\{v,y,z,r\}$, for
$G$. This contradiction shows that $r^*\sim s^*$ in $G^*$.
\par If
$y^*\nsim z^*$, then $y^*$ is of type (K), since $\Gamma_1(v)$ is
not homogeneous. Moreover, $s^*$ and $r^*$ are of type (1K), since
$\Gamma_2(v)$ is homogeneous. However $\Gamma_1(r)$ is a clique,
this contradiction shows that both edges $r^*s^*$ and $y^*z^*$
exist in $G^*$. Therefore, $y^*$ is of type (N), since
$\Gamma_1(v)$ is not homogeneous. Furthermore, $r^*$ and $s^*$ are
of type (1K), because $\Gamma_2(v)$ is homogeneous. Also, since
$\Gamma_2(v)$ is a clique, $v^*$ is of type (1K). As we see
before, when $y^*$ is of type (N), the other vertices of $G^*$ are
not of type (K), hence other vertices are of type (1). Moreover,
two none-adjacent vertices are not of the same type (K).
Consequently, $G^*$ satisfies the structure $G_6$.
 }\end{proof}
\begin{pro}\label{claim3-4}
If for each $u\in R_2(v)$ the set $M_2(u)=\emptyset$, $R_1(v)$ is
not homogeneous, and $|\Gamma^*_2(v^*)|=1$, then  $G^*$ satisfies
one of the structures $G_8$ to $G_{10}$.
\end{pro}
\begin{proof}{ Since $R_1(v)$ is not homogeneous, there exist a vertex $x\in
R_1(v)$ such that $M_2(x)\neq\emptyset$. If there is no edge
between $R^*_1(v^*)\cap M^*_1(x)$ and $R^*_1(v^*)\cap M^*_2(x)$ or
all vertices of $R^*_1(v^*)\cap M^*_1(x)$ are adjacent to all
vertices of $R^*_1(v^*)\cap M^*_2(x)$, then $|R^*_1(x)|\leq 3$,
because all distinct neighbors of the vertices in these two sets
are in each other. In the same manner as the proof of
Proposition~\ref{claim3-3}, we can see that each one of the sets
$R^*_1(v^*)\cap M^*_1(x)$ and $R^*_1(v^*)\cap M^*_2(x)$ has at
most two vertices. Since $R_2(v)$ is a clique and every vertex of
it is adjacent to all vertices of $R_1(v)$, if
$R_2(v)\neq\emptyset$, then all vertices of $R_2(v)$ are twins
with $v$, and so $v^*$ is of type (K). Also, if
$R_2(v)=\emptyset$, then $v^*$ and the only vertex of
$\Gamma^*_2(v^*)$, say $w^*$, are twins. Since $w^*\nsim v^*$, one
of them is of type (K). Note that $R_2(v)=\emptyset$ shows $v^*$
is not of type (K), consequently, $w^*$ is of type (K). In this
case, by symmetry of $v^*$ and $w^*$ (see Figure~\ref{6}(a)), we
can assume that $v^*$ is of type (K), and so all vertices of
$R_2(v)$ are twins with $v$, and in consequence, without loss of
generality we can assume that $R^*_2(v^*)=\emptyset$.
\par Now let both $R^*_1(v^*)\cap M^*_1(x)$
and $R^*_1(v^*)\cap M^*_2(x)$ have two vertices $\{a^*,b^*\}$ and
$\{y^*,z^*\}$, respectively, see Figure~\ref{6}(a). Since all
distinct neighbors of $z^*$ and $y^*$ are in $\{a^*,b^*\}$,
$\{a^*,b^*\}$ resolves $\{y^*,z^*\}$. Suppose that $v_1,v_2\in
v^*,~y\in y^*,~z\in z^*$, and $w\in w^*$. Then $a^*\cup b^*$
resolves $\{y,z\}$. Also, each coordinate of $r(y|x^*)$ and
$r(z|x^*)$ is $2$. While $r(v_1|x^*)=r(w|x^*)$ is entirely $1$,
$r(v_1|v^*\backslash\{v_1\})$ is completely $1$, and all
coordinates of $r(w|v^*\backslash\{v_1\})$ are $2$. Therefore,
$V(G)\backslash \{y,z,v_1,w\}$ is a resolving set for $G$, which
is impossible. Hence, at least one of the sets $R^*_1(v^*)\cap
M^*_1(x)$ or $R^*_1(v^*)\cap M^*_2(x)$ has one vertex.
\par If $R^*_1(v^*)\cap
M^*_2(x)$ has one vertex and $R^*_1(v^*)\cap M^*_1(x)$ has two
vertices, then $a^*\neq b^*$ and $y^*=z^*$. Thus, the only
distinct neighbor of $a^*$ and $b^*$ is $y^*$, and so $y^*$ is
adjacent to exactly one of them, say $a^*$. Now
$r(v_1|\{v_2,x,a\})=(1,1,1)$, $r(y|\{v_2,x,a\})=(1,2,1)$,
$r(b|\{v_2,x,a\})=(1,1,*)$, and $r(w|\{v_2,x,a\})=(2,1,1)$, where
$*=1$ or $2$. If $*=2$, then $\beta(G)\leq n-4$, therefore $*=1$.
It follows that $x^*$ and $b^*$ are twins and they are adjacent,
hence one of them is of type (N). Since $R^*_1(v^*)\cap M^*_1(x)$
is homogeneous, $a^*\sim b^*$ gives $b^*$ is of type (1K), and so
$x^*$ is of type (N). In this case, $V(G)\backslash\{x,a,y,w\}$ is
a resolving set for $G$, a contradiction.
\par If $R^*_1(v^*)\cap
M^*_1(x)$ has one vertex and $R^*_1(v^*)\cap M^*_2(x)$ has two
vertices, then $a^*=b^*$ and $y^*\neq z^*$. Hence, $a^*$ is
adjacent to exactly one of the vertices $z^*$ and $y^*$, say
$y^*$. Thus, $V(G)\backslash \{v_1,y,z,w\}$ is a resolving set for
$G$. This contradiction yields $|R^*_1(v^*)|\leq 3$. Since
$R_1(v)$ is not homogeneous, $|R^*_1(v^*)|\geq 2$. First, let
$|R^*_1(v^*)|=3$ and $R^*_1(v^*)=\{x^*,a^*,y^*\}$. Then $G^*$ is
as Figure~\ref{6}(b). If $a^*\nsim y^*$, then $x^*$ and $a^*$ are
adjacent twins, consequently, one of them, say $x^*$, is of type
(N). This gives the resolving set, $V(G)\backslash \{v_1,x,y,w\}$
for $G$, therefore $a^*\sim y^*$, and so $x^*$ and $y^*$ are
twins, hence, one of them, say $x^*$, is of type (K). If $y^*$ is
of type (KN), then $V(G)\backslash\{a,x,y,w\}$ is a resolving set
for $G$, thus $y^*$ is of type (1). By the same argument $w^*$ is
of type (1). Note that $a^*$ is of type (1K), otherwise
$V(G)\backslash\{v_1,y,a,w\}$ is a resolving set for $G$. Thus
$G^*$ satisfies the structure of $G_{10}$.
\par When
$|R^*_1(v^*)|=2$, two cases can be happened.\vspace{4mm}\\
1. All vertices of $R_1(v)\cap M_1(x)$ are twins with $x$. In this
case $G^*\cong C_4$ and $x^*=a^*$. Since $R_1(v)\cap M_1(x)$ is
not empty, $x^*$ is of type (K). If $y^*$ is of type (KN), then
$V(G)\backslash\{v_1,x,y,w\}$ is a resolving set for $G$, hence
$y^*$ is of type (1). By the same reason $w^*$ is of type (1),
consequently $G^*$ satisfies the structure
$G_9$.\vspace{4mm}\\
2. All vertices of $R_1(v)\cap M_2(x)$ are twins with $x$. In this
case $G^*$ is the kite and $x^*=y^*$. Since $R_1(v)\cap M_2(x)$ is
not empty, $x^*$ is of type (N). If $a^*$ is of type (N), then
$V(G)\backslash\{v_1,x,a,w\}$ is a resolving set for $G$, hence
$a^*$ is of type (1K). If $w^*$ is of type (KN), then
$V(G)\backslash\{v_1,x,a,w\}$ is a resolving set for $G$. Thus,
$G^*$ satisfies the structure $G_8$, which proves the proposition.
}
\end{proof}
Now, the proof of necessity is completed.
\begin{figure}[ht]\hspace{3.7cm}
\vspace*{5cm}\special{em:graph fig7.bmp} \vspace*{1.5cm}
\vspace{-3.3cm}
\caption{\label{6}\footnotesize
$|\Gamma^*_2(v^*)|=1$ and $R_1(v)$ is not homogeneous.}
\end{figure}

\subsection{\label{Sufficiency} Proof of Sufficiency}
In this section we prove that, if $G$ is a graph of order $n$ and
$diam(G)=2$ such that $G^*$ has one of the structures $G_1$ to
$G_{10}$ in Theorem~\ref{d=2}, then $\beta(G)=n-3$. In the
following we consider each structure $G_1$ to $G_{10}$, as shown
in Figure~\ref{G^*and dim=n-3}, separately. In each case, we
assume that $i\in i^*,~j\in j^*,~c\in c^*,~p\in p^*,~q\in q^*$,
and $u\in u^*$. \vspace{4mm}\\
{$\bf G_1$}. Since $G^*$ has three
vertices, by Proposition~\ref{B(G)>n-n*}, $\beta(G)\geq n-3$. On
the other hand, by Theorem~\ref{n-2}, $\beta(G)\neq n-2$, also $G$
is not a complete graph. Therefore, $\beta(G)\leq n-3$. Hence
$\beta(G)=n-3$. \vspace{4mm}\\
{$\bf G_2$}. Similar to above, we
deduce $\beta(G)=n-3$. \vspace{4mm}\\
{$\bf G_3$}. Let $H,~H_1,~H_2$, and $H_3$ be four graphs such that
their twin graphs are as $G_3$ in Figure~\ref{G^*and dim=n-3}.
Also assume that in $H^*$ the vertex $j^*$ is of type (N) and the
other vertices are of type (1), in $H^*_1$ both $i^*$ and $u^*$
are of type (K), $j^*$ is of type (N), and $p^*$ is of type (1),
in $H^*_2$ both $j^*$ and $p^*$ are of type (N), $i^*$ is of type
(K), and $u^*$ is of type (1), in $H^*_3$, $u^*$ is of type (1)
and other vertices are of type (N).
\par Thus, $H$ is an
induced subgraph of $G$ and $G$ is an induced subgraph of $H_t$,
for some $t$, $1\leq t\leq 3$. Now we get the metric dimension of
$H$ and $H_t$'s, for $1\leq t\leq 3$. Since in $H^*$, the vertex
$j^*$ is of type (N), each resolving set for $H$ contains at least
$|j^*|-1$ vertices of $j^*$. Moreover,
$r(u|j^*\backslash\{j\})=r(i|j^*\backslash\{j\})$, hence,
$j^*\backslash\{j\}$ is not a resolving set for $H$. It is easy to
check that $(j^*\backslash\{j\})\cup\{u\}$ is a resolving set, and
so a basis of $H$. Thus, $\beta(H)=n(H)-3$. Since in $H^*_1$ the
vertices $i^*,~j^*$, and $u^*$ are not of type (1), each resolving
set for $H_1$ contains at least $|i^*|-1,\-|j^*|-1$, and $|u^*|-1$
vertices of $i^*,\-j^*$, and $u^*$, respectively. On the other
hand $r(u|i^*\cup j^* \cup u^*\backslash\{i,j,u\})=r(i|i^*\cup j^*
\cup u^*\backslash\{i,j,u\})$, therefore $i^*\cup j^* \cup
u^*\backslash\{i,j,u\}$ does not resolve $H_1$. It is easy to see
that $i^*\cup j^* \cup u^*\backslash\{j,u\}$ resolves $H_1$,
hence, it is a basis of $H_1$, and consequently
$\beta(H_1)=n(H_1)-3$. By a same argument, we can see
$\beta(H_t)=n(H_t)-3$, $2\leq t\leq 3$. Now, since
$diam(H)=diam(G)=2$, by Corollary~\ref{sugraphH<K<G}, we have
$\beta(G)=n-3$. \vspace{4mm}\\
{$\bf G_4$}. It is clear that
$\beta(C_5)=2$. \vspace{4mm}\\
{$\bf G_5$}. Let $H$ and $R$ be two graphs such that their twin
graphs are as $G_5$, all vertices of $H^*$ are of type (1) and in
$R^*$ both vertices $p^*$ and $q^*$ are of type (1) and other
vertices are of type (K). Therefore, $H$ is an induced subgraph of
$G$ and $G$ is an induced subgraph of $R$. It is clear that
$\beta(H)=2=n(H)-3$. Each resolving set for $R$ contains at least
$|i^*|-1,~|j^*|-1$, and $|u^*|-1$ vertices from $i^*,\-j^*$, and
$u^*$, respectively. Let $W=i^*\cup j^* \cup
u^*\backslash\{i,j,u\}$. Then $r(i|W)=r(j|W)=r(u|W)$, hence, $W$
is not a resolving set for $R$. Also adding one of the vertices
$i,\-j$, and $u$ to $W$, can not provide a resolving set for $R$,
because $\{i,j,u\}$ is a clique in $R$. Since $diam(R)=2$, neither
$p$ nor $q$ can not resolve more than two vertices from the set
$\{i,j,u\}$. Thus, $\beta(R)\geq |W|+2=n(R)-3$. Since $R$ is not
complete graph and none of the graphs in Theorem~\ref{n-2},
$\beta(R)\leq n(R)-3$. Hence, $\beta(R)=n(R)-3$. Since
$diam(H)=diam(G)=2$,
Corollary~\ref{sugraphH<K<G} yields $\beta(G)=n-3$. \vspace{4mm}\\
{$\bf G_6$}. Let $H,~H_1,~H_2$, and $H_3$ be four graphs with twin
graphs  as $G_6$ in Figure~\ref{G^*and dim=n-3}. Moreover, assume
that, all vertices of $H^*$ are of type (1), in $H^*_1$ the vertex
$i^*$ is of type (N) and the other vertices are of type (1), in
$H^*_2$ both $p^*$ and $q^*$ are of type (1) and the other
vertices are of type (K), in $H^*_3$ both $p^*$ and $u^*$ are of
type (1) and the other vertices are of type (K). Hence, $H$ is an
induced subgraph of $G$ and $G$ is an induced subgraph of $H_t$,
for some $t$, $1\leq t\leq 3$. It is clear that
$\beta(H)=2=n(H)-3$. Each resolving set for $H_1$ contains at
least $|i^*|-1$ vertices from $i^*$. But $r(u|i^*\backslash
\{i\})=r(j|i^*\backslash \{i\})=r(p|i^*\backslash
\{i\})=r(q|i^*\backslash \{i\})$ and there is no vertex of set
$\{u,j,p,q\}$ such that adding it to $i^*\backslash\{i\}$ provides
a resolving set for $H_1$, hence $\beta(H_1)\geq
|i^*|+1=n(H_1)-3$. Since $H_1$ is not complete graph,
Theorem~\ref{n-2} gives $\beta(H_1)\leq n(H_1)-3$, and so
$\beta(H_1)=n(H_1)-3$. Each resolving set for $H_2$ contains at
least $|i^*|-1,\-|j^*|-1$, and $|u^*|-1$ vertices from
$i^*,\-j^*$, and $u^*$, respectively. Assume $W=i^*\cup j^*\cup
u^*\backslash\{i,j,u\}$. It follows that $r(i|W)=r(j|W)=r(u|W)$,
hence $W$ does not resolve $H_2$. It is easy to check that to
provide a resolving set for $H_2$ we need to add at least two
vertices from $V(H_2)-W$ to $W$. Thus, $\beta(H_2)\geq
|W|+2=n(H_2)-3$. On the other hand, Theorem~\ref{n-2} implies
$\beta(H_2)\leq n(H_2)-3$, and consequently $\beta(H_2)=n(H_2)-3$.
By the same way, $\beta(H_3)=n(H_3)-3$. Since $diam(H)=diam(G)=2$,
Corollary~\ref{sugraphH<K<G} implies $\beta(G)=n-3$.
\vspace{4mm}\\
{$\bf G_7$}. Let $H$ and $R$ be two graphs such that their twin
graphs are as $G_7$ in Figure~\ref{G^*and dim=n-3}. Moreover,
assume that in $H^*$ the vertex $u^*$ is of type (K) and the other
vertices are of type (1), and in $R^*$ both vertices $p^*$ and
$q^*$ are of type (1) and other vertices are of type (K).
Therefore, $H$ is an induced subgraph of $G$ and $G$ is an induced
subgraph of $R$. It is clear that each resolving set for $H$
contains at least $|u^*|-1$ vertices of $u^*$. Moreover,
$r(u|u^*\backslash\{u\})=r(i|u^*\backslash\{u\})=r(j|u^*\backslash\{u\})$
and to provide a resolving set for $H$, we must add at least two
vertices from the set $\{u,i,j,p,q\}$ to $u^*\backslash\{u\}$,
hence $\beta(H)\geq |u^*|+1=n(H)-3$. Also by Theorem~\ref{n-2},
$\beta(H)\leq n(H)-3$, thus $\beta(H)=n(H)-3$. Since $i^*,\-j^*$,
and $u^*$ are not of type (1) in $R^*$, each resolving set for $R$
contains at least $|i^*|-1,~|j^*|-1$, and $|u^*|-1$ vertices of
$i^*,\-j^*$, and $u^*$, respectively. For $W=i^*\cup j^* \cup
u^*\backslash\{i,j,u\}$, we have $r(i|W)=r(j|W)=r(u|W)$, hence $W$
is not a resolving set for $R$. To provide a resolving set for
$R$, we need to add at least two vertices from the set
$\{u,i,j,p,q\}$ to $W$, and consequently, $\beta(R)\geq
|W|+2=n(R)-3$. Theorem~\ref{n-2} shows that $\beta(R)\leq n(R)-3$,
hence $\beta(R)=n(R)-3$. Since $diam(H)=diam(G)=2$,
Corollary~\ref{sugraphH<K<G} yields $\beta(G)=n-3$. \vspace{4mm}\\
{$\bf G_8$}. Let $H$ and $R$ be two graphs such that their twin
graphs are as $G_8$ in Figure~\ref{G^*and dim=n-3}, where in $H^*$
both $j^*$ and $p^*$ are of type (1), $u^*$ is of type (K) and
$i^*$ is of type (N), and in $R^*$ both vertices $j^*$ and $u^*$
are of type (K), $i^*$ is of type (N) and $p^*$ is of type (1).
Then $G$ is one of the graphs $H$ or $R$. Theorem~\ref{n-2} shows
that $\beta(H)\leq n(H)-3$ and $\beta(R)\leq n(R)-3$. Note that
each resolving set for $H$ contains at least $|u^*|-1$ and
$|i^*|-1$ vertices from $u^*$ and $i^*$, respectively. If
$S=(i^*\cup u^*)\backslash\{i,u\}$, then $r(u|S)=r(j|S)$.
Therefore, $\beta(H)\geq |S|+1=n(H)-3$, and so $\beta(H)=n(H)-3$.
It is clear that every resolving set for $R$ contains at least
$|i^*|-1,~|j^*|-1$, and $|u^*|-1$ vertices of $i^*,~j^*$, and
$u^*$, respectively. Let $W=i^*\cup j^* \cup
u^*\backslash\{i,j,u\}$. Then $r(j|W)=r(u|W)$, hence $W$ is not a
resolving set for $R$, and so $\beta(R)\geq |W|+1=n(R)-3$. It
follows that $\beta(R)=n(R)-3$. Consequently, $\beta(G)=n(G)-3$.
\vspace{4mm}\\
{$\bf G_9$}. Let $G^*$ be as $G_9$ in
Figure~\ref{G^*and dim=n-3}, where $i^*$ and $u^*$ are of type (K)
in $G^*$, and the other vertices are of type (1). Each resolving
set for $G$ contains at least $|i^*|-1$ and $|u^*|-1$ vertices of
$i^*$ and $u^*$, respectively. For $W=i^*\cup
u^*\backslash\{i,u\}$, we have $r(i|W)=r(u|W)$, hence $W$ is not a
resolving set for $G$, and so $\beta(G)\geq |W|+1=n(G)-3$.
Theorem~\ref{n-2} implies $\beta(G)\leq n(G)-3$. Consequently,
$\beta(G)=n(G)-3$. \vspace{4mm}\\
{$\bf G_{10}$}. Let $H$ and $R$ be two graphs such that their twin
graphs are as $G_{10}$ in Figure~\ref{G^*and dim=n-3},
furthermore, in $H^*$ both $u^*$ and $c^*$ are of type (K) and
other vertices are of type (1), and in $R^*$ both vertices $i^*$
and $p^*$ are of type (1) and other vertices are of type (K). Then
$G$ is one of the graphs $H$ or $R$. Theorem~\ref{n-2} shows that
$\beta(H)\leq n(H)-3$ and $\beta(R)\leq n(R)-3$. Note that each
resolving set for $H$ contains at least $|u^*|-1$ and $|c^*|-1$
vertices from $u^*$ and $c^*$, respectively. If $S=c^*\cup
u^*\backslash\{c,u\}$, then $r(u|S)=r(j|S)=r(c|S)$, therefore $S$
does not resolve $H$. To provide a resolving set for $H$, we need
to add at least two vertices from the set $\{u,i,j,c,p\}$ to $S$,
and so $\beta(H)\geq |S|+2=n(H)-3$, hence $\beta(H)=n(H)-3$. It is
clear that every resolving set for $R$ contains at least
$|u^*|-1,\-|j^*|-1$, and $|c^*|-1$ vertices from $u^*,\-j^*$, and
$c^*$, respectively. Let $W=u^*\cup j^* \cup
c^*\backslash\{u,j,c\}$. Hence, $r(u|W)=r(j|W)=r(c|W)$, therefore
$W$ is not a resolving set for $R$. Clearly, to provide a
resolving set for $R$, we must add at least two vertices from the
set $\{u,i,j,c,p\}$ to $W$, hence $\beta(R)\geq |W|+2=n(R)-3$,
thus $\beta(R)=n(R)-3$. Consequently, $\beta(G)=n(G)-3$.\\ This
completes the proof of the sufficiency of
Theorem~\ref{d=2}.\hfill{$\rule{2mm}{2mm}$}

\end{document}